\documentclass[10pt, a4paper, reqno]{amsart}
\usepackage{amscd}
\usepackage{dsfont, pifont}

\usepackage{amsmath, upgreek}
\usepackage{amstext}
\usepackage{amsthm}
\usepackage{amssymb}
\usepackage{amsopn}
\usepackage{amssymb}
\usepackage{amsxtra}
\usepackage{amsfonts}
\usepackage{fancyhdr}
\usepackage{paralist, epsfig}
\usepackage[mathcal]{euscript}
\usepackage[english]{babel}
\usepackage[latin1]{inputenc}
\usepackage{bbold}
\usepackage{tkz-graph}

\newcommand*\quot[2]{{^{\textstyle #1}\big/_{\textstyle #2}}}

\makeatletter
\g@addto@macro\normalsize{%
  \setlength\abovedisplayskip{10pt}
  \setlength\belowdisplayskip{10pt}
  \setlength\abovedisplayshortskip{5pt}
  \setlength\belowdisplayshortskip{8pt}
}

\allowdisplaybreaks

\newtheoremstyle{normal}
{5pt}
{5pt}
{\normalfont}
{}
{\bfseries}
{}
{0.4em}
{\bfseries{\thmname{#1}\thmnumber{ #2}.\thmnote{ \hspace{0.5em}(#3)\newline}}}

\newtheoremstyle{kursiv}
{5pt}
{5pt}
{\itshape}
{}
{\bfseries}
{}
{0.4em}
{\bfseries{\thmname{#1}\thmnumber{ #2}.\thmnote{ \hspace{0.5em}(#3)\newline}}}

\usepackage[indentafter]{titlesec}
\titleformat{name=\section}{}{\thetitle.}{0.8em}{\centering\scshape}
\titleformat{name=\subsection}{}{\thetitle.}{0.8em}{\centering\scshape}

\theoremstyle{kursiv}

\theoremstyle{normal}

\newtheorem{thm}{Theorem}[section]

\newtheorem{ex}[thm]{Example}
\newtheorem{rem}[thm]{Remark}
\newtheorem{cor}[thm]{Corollary}
\newtheorem{lem}[thm]{Lemma}
\newtheorem{prop}[thm]{Proposition}
\newtheorem{dfn}[thm]{Definition}

\renewcommand{\epsilon}{\varepsilon}

\newcommand{\T}{(T(t))_{t\geqslant0}}

\newcommand{\supp}{\operatorname{supp}\nolimits}

\renewcommand{\Re}{\operatorname{Re}\nolimits}

\renewcommand{\theta}{\vartheta}
\newcommand{\Bigosum}[2]{\ensuremath{\mathop{\textstyle\bigoplus}_{#1}^{#2}}}

\renewcommand{\bar}[1]{\overline{#1}}
\newcommand{\ran}{\operatorname{ran}\nolimits}

\newcommand{\e}{\operatorname{e}\nolimits}

\newcommand{\Cnull}{\operatorname{C_0}\nolimits}
\newcommand{\diag}{\operatorname{diag}\nolimits}

\usepackage{setspace}
\usepackage{color}

\definecolor{grey}{gray}{.3}

\renewcommand{\star}{*}

\newcommand{\rank}{\operatorname{rk}}
\newcommand{\Ls}{\operatorname{L}}
\newcommand{\AC}{\operatorname{AC}_{\operatorname{loc}}}
\newcommand{\dd}{\mathrm{d}}

\begin{document}
\allowdisplaybreaks

\title{Well-posedness of a class of hyperbolic partial\\differential equations on the semi-axis}

\author{Birgit Jacob\hspace{0.5pt}\MakeLowercase{$^{\text{1}}$} and Sven-Ake Wegner\hspace{0.5pt}\MakeLowercase{$^{\text{2}}$}}

\renewcommand{\thefootnote}{}
\hspace{-1000pt}\footnote{\hspace{5.5pt}2010 \emph{Mathematics Subject Classification}: Primary 93D15; Secondary 47D06.\vspace{1.6pt}}

\hspace{-1000pt}\footnote{\hspace{5.5pt}\emph{Key words and phrases}: $\Cnull$-semigroup, hyperbolic pde, port-Hamiltonian system, well-posedness, pde's on networks. \vspace{1.6pt}}

\hspace{-1000pt}\footnote{\hspace{0pt}$^{1}$\,University of Wuppertal, School of Mathematics and Natural Sciences, Gau\ss{}stra\ss{}e 20, 42119 Wuppertal, Germany,\linebreak\phantom{x}\hspace{1.2pt}Phone: +49\hspace{1.2pt}(0)\hspace{1.2pt}202\hspace{1.2pt}/\hspace{1.2pt}439\hspace{1.2pt}-\hspace{1.2pt}2527, Fax: +49\hspace{1.2pt}(0)\hspace{1.2pt}202\hspace{1.2pt}/\hspace{1.2pt}439\hspace{1.2pt}-\hspace{1.2pt}3724, E-Mail: bjacob@uni-wuppertal.de.\vspace{1.6pt}}

\hspace{-1000pt}\footnote{\hspace{0pt}$^{2}$\,Corresponding author: Teesside University, School of Science, Engineering \&{} Design, Stephenson Building, Middles-\linebreak\phantom{x}\hspace{1.2pt}brough, TS1\;3BX, United Kingdom, phone: +44\,(0)\,1642\:434\:82\:00, e-mail: s.wegner@tees.ac.uk.\vspace{1.6pt}}


\begin{abstract}
In this article we study a class of hyperbolic partial differential equations of order one on the semi-axis. The so-called port-Hamiltonian systems cover for instance the wave equation and the transport equation, but also networks of the aforementioned equations fit into this framework. Our main results firstly characterize the boundary conditions which turn the corresponding linear operator into the generator of a strongly continuous semigroup. Secondly, we equip the equation with inputs (control) and outputs (observation) at the boundary and prove that this leads to  a well-posed boundary control system. We illustrate our results via an example of coupled transport equations on a network, that allows to model transport from and to infinity. Moreover, we study a vibrating string of infinite length with one endpoint. Here, we show that our results allow to treat cases where the physical constants of the string tend to zero at infinity.
\end{abstract}

\maketitle


\section{Introduction}\label{SEC:INTRO}\smallskip

Let $n\geqslant1$ be a fixed integer, let $P_1\in\mathbb{C}^{n\times n}$ be Hermitian and invertible and let $P_0\in\mathbb{C}^{n\times n}$ be arbitrary. Let, for the moment, $\mathcal{H}\colon[0,\infty)\rightarrow\mathbb{C}^{n\times n}$ be continuous such that $\mathcal{H}(\xi)$ is positive and Hermitian for all $\xi\in[0,\infty)$. We consider the port-Hamiltonian partial differential equation
\begin{equation}\label{PHS-INTRO}
\frac{\partial x}{\partial t}(\xi,t)=P_1\frac{\partial}{\partial \xi}(\mathcal{H}(\xi)x(\xi,t))+P_0\hspace{0.5pt}\mathcal{H}(\xi)x(\xi,t)\:\text{ for }\:\xi\in[0,\infty),\:t\geqslant0
\end{equation}
on the semi-axis with the initial condition $x(\xi,0)=x_0(\xi)$ for $\xi\in[0,\infty)$. The matrix-valued function $\mathcal{H}$ is referred to as the Hamiltonian or the Hamiltonian density matrix.  The relevant boundary conditions for the port-Hamiltonian partial differential equation are given by
\begin{equation}\label{BDY-INTRO}
W_B\cdot\mathcal{H}(0)x(0,t)=0\:\text{ for }\:t\geqslant0
\end{equation}
where $W_B\in\mathbb{C}^{n_-\times n}$ is a matrix of rank $n_-$ and $n_-$ is the number of negative eigenvalues of $P_1$. The reason why exactly $n_-$ is the correct number of boundary conditions stems from a diagonalization technique we will sketch below and explain with all details in Section \ref{SEC:GEN}.

\medskip

In this paper we approach the partial differential equation above with concepts from operator and systems theory. For this reason we interpret \eqref{PHS-INTRO} as an abstract differential equation in a Hilbert space and a natural choice for the latter is the weighted $\Ls^2$-space 
\begin{equation}\label{SPACE-INTRO}
\Ls^2_{\mathcal{H}}(0,\infty) = \bigl\{x\colon [0,\infty)\rightarrow\mathbb{C}^n\:;\:x \text{ measurable and } \|x\|_{\Ls^2_{\mathcal{H}}(0,\infty)}^2=\int_{0}^{\infty}x(\xi)^{\star}\mathcal{H}(\xi)x(\xi)\dd\xi<\infty\bigr\}
\end{equation}
endowed with the scalar product $\langle{}\cdot,\cdot\rangle_{\Ls^2_{\mathcal{H}}(0,\infty)}=\langle{}\cdot,\mathcal{H}\cdot\rangle_{\Ls^2(0,\infty)}$. On a finite interval the corresponding definition has been used successfully in the past, see e.g., van der Schaft, Maschke \cite{SM}, Le Gorrec et.~al.~\cite{GZM2005}, Villegas  \cite{V2007}, Villegas et.~al.~\cite{VZGM2009}, Zwart et.~al.~\cite{ZGMV2010}, Engel \cite{Engel2013}, Augner, Jacob \cite{AJ2014}, Jacob, Zwart \cite{JZ}, Wegner \cite{BT}. On a finite interval $I$ our assumptions on $\mathcal{H}$ imply automatically that there are constants $m$, $M>0$ such that $m|\zeta|^2\leqslant \zeta^{\star}\mathcal{H}(\xi)\zeta\leqslant M|\zeta|^2$ holds for all $\xi\in I$ and all $\zeta\in\mathbb{C}^n$. From this it follows that $\Ls^2_{\mathcal{H}}(I)=\Ls^2(I)$ holds in the sense of equal linear spaces with equivalent norms. If one restricts the attention to classifying when a contraction semigroup is generated, one can even assume without loss of generality that $\mathcal{H}\equiv{}\mathbb{1}$ holds \cite[Section 7]{JZ}. On the other hand recent results by Jacob et.~al.~\cite{JMZ2015} show that for the case of possibly non-contractive semigroups the latter is not true. Although $\Ls^2_{\mathcal{H}}(I)=\Ls^2(I)$ holds in their setting with equivalent norms, and their proofs make use of this fact, it turns out that generation for the given $\mathcal{H}\not\equiv\mathbb{1}$ and generation with $\mathcal{H}\equiv\mathbb{1}$ are not equivalent. In this paper we are additionally confronted with the fact that on a non-compact domain the function $\mathcal{H}$ can be continuous without being bounded or being bounded away from zero, from whence it follows that neither $\Ls^2_{\mathcal{H}}(0,\infty)=\Ls^2(0,\infty)$ holds in the sense of linear spaces, nor that we have any estimates between the norms $\|\cdot\|_{\Ls^2_{\mathcal{H}}(0,\infty)}$ and $\|\cdot\|_{\Ls^2(0,\infty)}$ a priori.

\medskip

Having fixed an appropriate space \eqref{SPACE-INTRO}, we associate with the equation \eqref{PHS-INTRO} the operator
\begin{equation}\label{OP-INTRO}
\begin{array}{rcl}
Ax&=&\displaystyle P_1(\mathcal{H}x)'+P_0\hspace{0.5pt}\mathcal{H}x\\[9pt]
D(A)&=&\bigl\{x\in \Ls^2_{\mathcal{H}}(0,\infty)\:;\: (\mathcal{H}x)'\in\Ls^2_{\mathcal{H}}(0,\infty) \text{ and }W_B\cdot\mathcal{H}(0)x(0)=0\bigr\}
\end{array}
\end{equation}
where we encode the boundary conditions \eqref{BDY-INTRO} in its domain and understand $(\mathcal{H}x)'\in\Ls^2_{\mathcal{H}}(0,\infty)$ in the sense of a weak derivative that can be represented by an $\Ls^1_{\operatorname{loc}}$-function belonging to $\Ls^2_{\mathcal{H}}(0,\infty)$, cf.~Section \ref{SEC:PREP} for details. Notice that the above is at least well-defined if $P_0=0$ or if $\mathcal{H}$ is bounded. Our main results on generation, see Theorem \ref{PHS-THM} and Corollary \ref{PHS-COR}, will address precisely these two cases and characterize in terms of a matrix condition formulated via $W_B$, when $A\colon D(A)\rightarrow\Ls^2_{\mathcal{H}}(0,\infty)$ generates a $\Cnull$-semigroup. Our proof is inspired by a method used by Zwart et.~al.~\cite{ZGMV2010}, see also \cite[Section 13]{JZ} and \cite{JMZ2015}, which relies on a diagonalization
\begin{equation}\label{DIAG-INTRO}
P_1\mathcal{H}(\xi)=S(\xi)^{-1}\Delta(\xi)S(\xi)=S(\xi)^{-1}{\textstyle\begin{bmatrix}\Lambda(\xi) & 0\\0&\Theta(\xi)\end{bmatrix}}S(\xi)
\end{equation}
via matrix-valued functions $S$, $S^{-1}\colon[0,\infty)\rightarrow\mathbb{C}^{n\times n}$. The values of the functions $\Lambda\colon[0,\infty)\rightarrow\mathbb{C}^{n_+\times n_+}$ and $\Theta\colon[0,\infty)\rightarrow\mathbb{C}^{n_-\times n_-}$ are diagonal and positive, resp.~negative, matrices where $n_+$ is the number of positive and $n_-$ is the number of negative eigenvalues of $P_1\mathcal{H}(\xi)$. This number is by Sylvester's law of inertia independent of $\xi\in[0,\infty)$. The diagonal operator can now be treated by a divide-and-conquer strategy as each of its components generates a one-dimensional \textquotedblleft{}weighted shift semigroup\textquotedblright{}. From this it can be seen why the appropriate number of boundary conditions is $n_-$: it is the number of right shifts in the diagonal operator, each of which produces one boundary condition at zero. The $n_+$ left shifts produce no boundary conditions. In a second step the Weiss Theorem \cite{Weiss}, see Section \ref{INTERLUDE}, is applied to get all linear boundary conditions for the diagonal operator. This leaves us with a generation result on an $\Ls^2$-space weighted with $|\Delta|=(\Delta^{\star}\Delta)^{1/2}$. In the final step we need to pullback the latter to the $\Ls^2$-space weighted with the Hamiltonian $\mathcal{H}$ which can be achieved by interpreting $S\colon \Ls^2_{\mathcal{H}}(0,\infty)\rightarrow \Ls^2_{|\Delta|}(0,\infty)$, $x\mapsto Sx$, as a transformation of variables. All three steps require technical assumptions on $\mathcal{H}$, $\Delta$, $S$ and $S^{-1}$ which are in full detail given in Section \ref{SEC:GEN}. 

\medskip

Once the question of characterizing the generator property is settled, we add to the partial differential equation \eqref{PHS-INTRO} an input $u=u(t)$ and an output $y=y(t)$, i.e., we consider
\begin{equation}\label{INOUT-INTRO}
\begin{array}{rcl}
u(t)&=&W_{B,1}\mathcal{H}(0)x(0,t)\\[4.5pt]
0&=&W_{B,2}\mathcal{H}(0)x(0,t)\\[4.5pt]
y(t)&=&W_C\mathcal{H}(0)x(0,t)
\end{array}
\end{equation}
where we assume $W_C\in\mathbb{C}^{q\times n}$ and $W_B=[W_{B,1}\;\;W_{B,2}]^T$ with $W_{B,1}\in\mathbb{C}^{p\times n}$ to allow that not all boundary conditions are subject to a control but some have just zero input. In Section \ref{SEC:WP} we show that \eqref{PHS-INTRO} and \eqref{INOUT-INTRO} give rise to a boundary control system which is well-posed. That is, for every $\tau>0$ there exists $m_{\tau}>0$ such that for every $x_0\in\Ls^2_{\mathcal{H}}(0,\infty)$ with $(\mathcal{H}x)'\in\Ls^2_{\mathcal{H}}(0,\infty)$ and every $u\in\operatorname{C}^2([0,\tau],\mathbb{C}^p)$ with $u(0)=W_{B,1}\mathcal{H}(0)x_0(0)$ there is a unique classical solution $x=x(\xi,t)$ such that 
\begin{equation}
\|x(\cdot,\tau)\|_{\Ls^2_{\mathcal{H}}(0,\infty)}^2+\int_0^{\tau}\|y(t)\|^2\dd t\leqslant m_{\tau}\Big(\|x_0\|_{\Ls^2_{\mathcal{H}}(0,\infty)}^2+\int_0^{\tau}\|u(t)\|^2\dd t\Big)
\end{equation}
holds. For details on the notion of well-posedness see, e.g., Tucsnak, Weiss \cite{TW2014}, Staffans \cite{Staffans2005} or Jacob, Zwart \cite[Chapter 13]{JZ}.

\medskip

Finally, in Section \ref{SEC:EX} we discuss the \textquotedblleft{}weighted transport equation\textquotedblright{}, a network of transport equations, each defined on $[0,\infty)$, which are coupled at a central node, and a vibrating string of infinite length. We mention that the latter examples as well as our main results are related to recent results by Jacob, Kaiser \cite[Section 2.2 and Section 5]{JK} where the case of contraction semigroups associated with port-Hamiltonian partial differential equations is studied.

\bigskip


\section{Preparation}\label{SEC:PREP}\smallskip

In this section we first introduce weighted $\Ls^2$-spaces of scalar valued functions in one variable. Notice, that in contrast to previous results on port-Hamiltonian partial differential equations \eqref{PHS-INTRO} in this article the weighted spaces will not necessarily be isomorphic to the unweighted $\Ls^2$-space. Secondly we will repeat well-known facts about the relation of absolutely continuous functions and functions with an integrable weak derivative. Here, the relationship is presented from the point of view of weighted, instead of classical, $\Ls^2$-spaces. The third objective of this section is to prove a technical lemma that we will need later in the proofs of our generation results.

\smallskip

For an interval $I\subseteq\mathbb{R}$ and a continuous function $w\colon I\rightarrow(0,\infty)$ we consider
$$
\Ls^2_{w}(I) := \bigl\{x\colon I\rightarrow\mathbb{C}\:;\:x \text{ measurable and } \|x\|_{\Ls^2_{w}(I)}^2=\int_{I}w(\xi)|x(\xi)|^2\dd\xi<\infty\bigr\}
$$
which is a Hilbert space with respect to the scalar product
$$
\langle{}x,y\rangle{}_{\Ls^2_{w}(I)}:=\int_{I}\bar{x(\xi)}w(\xi)y(\xi)\dd\xi.
$$
The unweighted $\Ls^2$-spaces we obtain as the specialization $\Ls^2(I)=\Ls^2_{\mathbb{1}}(I)$. At some point we consider bounded intervals $[a,b]\subseteq I$ and write $\Ls^2_{w}(a,b)$ instead of $\Ls^2_{w|{(a,b)}}(a,b)$ to simplify our notation.  We note that the spaces $\Ls^2_{w}(a,b)$ and $\Ls^2(a,b)$ are equal as linear spaces and that their norms $\|\cdot\|_{\Ls^2_{w}(a,b)}$ and $\|\cdot\|_{\Ls^2(a,b)}$ are equivalent in view of $0<\inf_{\xi\in[a,b]}w(\xi)\leqslant\sup_{\xi\in [a,b]}w(\xi)<\infty$. Notice that for unbounded $I$ the spaces $\Ls^2_{w}(I)$ and $\Ls^2(I)$ need not to be equal and no inclusion is valid a priori.

\medskip

It is well-known that the operator $Ax=x'$ generates the shift semigroup when we consider the latter as an operator $A\colon D(A)\rightarrow \Ls^2(I)$ with
\begin{equation}\label{EQ-D}
D(A)=\bigl\{x\in \Ls^2(I)\:;\:x\in \operatorname{AC}(I)\text{ and }x'\in \Ls^2(I)\bigr\}=\bigl\{x\in \Ls^2(I)\:;\:x'\in \Ls^2(I)\bigr\}
\end{equation}
on a bounded interval $I\subseteq\mathbb{R}$. Here, in the first set, $x'$ stands for the derivative almost everywhere. In the second set, $x'$ is the distributional derivative and writing $x'\in \Ls^2(I)$ means firstly that $x'$ is a regular distribution, i.e., $x'\in \Ls^1_{\text{loc}}(I)$, and secondly that it belongs to $\Ls^2(I)$. Below we stick to using the notion of absolutely continuous functions, but a characterization as above is also true in the weighted case, see Lemma \ref{LEM-DIS}. We firstly recall the following. A function $y\colon I\rightarrow\mathbb{C}$ is locally absolutely continuous if $y|_{[a,b]}\colon [a,b]\rightarrow\mathbb{C}$ is absolutely continuous for every $[a,b]\subseteq I$. We write
$$
\AC(I):=\bigl\{x\colon I\rightarrow\mathbb{C}\:;\: x \text{ is locally absolutely continuous}\,\bigr\}
$$
for the space of all locally absolutely continuous functions. The following lemma repeats the well-known characterization of the elements of $\AC(I)$ in terms of the fundamental theorem of calculus.

\smallskip

\begin{lem}\label{AC-LEM} Let $I\subseteq\mathbb{R}$ be an interval. The function $y\colon I\rightarrow\mathbb{C}$ is locally absolutely continuous if and only if the following three conditions are satisfied.
\begin{compactitem}\vspace{3pt}
\item[(i)] $y$ is continuous on $I$,\vspace{5pt}

\item[(ii)] $y$ is differentiable almost everywhere in $I$ with $y'\in \Ls^1_{\text{loc}}(I)$,\vspace{2pt}

\item[(iii)] for all $\xi,\,\eta\in I$ we have $y(\xi)-y(\eta)=\int_{\eta}^{\xi}y'(\zeta)\dd\zeta$.\vspace{3pt}
\end{compactitem}
\end{lem}

\begin{proof} It is enough to repeat the arguments of the real-valued case, see, e.g., Leoni \cite[Chapter 3]{GL}.
\end{proof}

\smallskip

Now we can state the following analogue of \eqref{EQ-D} for weighted $\Ls^2$-spaces over possibly unbounded intervals.

\smallskip

\begin{lem}\label{LEM-DIS} Let $I\subseteq\mathbb{R}$ be an interval and let $w\colon I\rightarrow(0,\infty)$ be continuous. For $x\in \Ls^2_w(I)$ the following are equivalent.
\begin{compactitem}\vspace{3pt}
\item[(i)] We have $wx\in \AC(I)$ and $(wx)'\in \Ls^2_w(I)$ in the sense of a derivative almost everywhere.\vspace{3pt}
\item[(ii)] We have $(wx)'\in \Ls^2_w(I)$ in the sense of a distributional derivative.
\end{compactitem}
\end{lem}
\begin{proof}Firstly we notice that a function $x\in \Ls^2_w(I)$ is in $\AC(I)$ if and only if it is in $\AC(\hspace{0.3pt}{\displaystyle\mathop{I}^{\circ}}\hspace{0.5pt})$. Thus, we may assume that $I$ is open and use the common textbook definition of $\mathcal{D}'(I)$, see, e.g., Folland \cite{Folland}.

\smallskip

(i)\,$\Rightarrow$\,(ii)~If $wx$ is locally absolutely continuous, then its almost everywhere defined derivative coincides with the distributional derivative. This shows that the latter is a regular distribution which then belongs to $\Ls^2_{w}(I)$ by assumption.

\smallskip

(ii)\,$\Rightarrow$\,(i)~Let $x\in \Ls^2_{w}(I)$ be given and assume that $(wx)'\in \Ls^1_{\text{loc}}(I)\subseteq\mathcal{D}'(I)$ even belongs to $\Ls^2_{w}(I)$. We need to show that $wx$ has a locally absolutely continuous representative $y$. We put $J_n:=(-n,n)\cap I$ for $n\geqslant0$. Then $wx|_{J_n}$ and $(wx)'|_{J_n}$ belong to $\Ls^2(J_n)$ and thus to $\operatorname{H}^2(J_n)$ which allows to select a uniquely determined locally absolutely continuous representative $y_n\colon J_n\rightarrow\mathbb{C}$ of $wx|_{J_n}$, see, e.g., Brezis \cite[Theorem 8.2 and Remark 5 on p.~204]{B}. We can therefore define $y\colon I\rightarrow\mathbb{C}$ via $y(\xi)=y_n(\xi)$ for $\xi\in J_n$ which is well-defined, belongs to $\AC(I)$, coincides with $wx$ in $\Ls^1_{\text{loc}}$, and thus belongs to $\Ls^2_{w}(I)$.
\end{proof}

\smallskip

The next lemma will be crucial in the proofs of all generation results that we will discuss below. Observe that in the unweighted case, i.e., $w\equiv 1$, the classical Barb{\u a}lat lemma, see, e.g., Farkas, Wegner \cite[Theorem 5]{FW}, shows that
$$
\lim_{\xi\rightarrow\pm\infty}x(\xi)=\lim_{\xi\rightarrow\pm\infty}(wx)(\xi)=0
$$ 
holds under the assumptions of Lemma \ref{BAR-LEM}. In the case of an arbitrary weight $w$ we get at least that $wx$ is bounded. If $w$ is bounded, we recover also that the two limits are zero. The proof is an adaption of \cite[Lemma 6]{FW} and Tao \cite{Tao}.

\smallskip

\begin{lem}\label{BAR-LEM} Let $x\in \Ls^2_{w}(\mathbb{R})$ be given such that $w{}x\in \AC(\mathbb{R})$ and $(w{}x)'\in \Ls^2_{w}(\mathbb{R})$. Then $w{}x$ is bounded. If $w$ is bounded, then $wx$ vanishes at infinity.
\end{lem} 
\begin{proof}1.~With $w{}x$ also $|w{}x|^2$ belongs to $\AC(\mathbb{R})$. Therefore we get that
$$
|(w{}x)(\xi)|^2-|(w{}x)(\eta)|^2 = \int_{\eta}^{\xi}\frac{\dd}{\dd\zeta}|(w{}x)(\zeta)|^2\dd\zeta
$$
holds for all $\xi,\,\eta\in\mathbb{R}$. We compute $\frac{\dd}{\dd\zeta}|(w{}x)(\zeta)|^2 = (wx)'(\zeta)\overline{wx}(\zeta)+\overline{wx}'(\zeta)(wx)(\zeta)= 2\Re(\overline{w{}x}(\zeta)(w{}x)'(\zeta))$ and get the estimate 
\begin{eqnarray*}
|(w{}x)(\xi)|^2 & \leqslant &  |(w{}x)(0)|^2+\Bigl|\int_0^{\xi} 2\Re(\overline{w{}x}(\zeta)(w{}x)'(\zeta))\dd\zeta\Bigr|\\
& \leqslant & |(w{}x)(0)|^2+2\int_0^{\xi}|\langle{}(w{}^{1/2}x)(\zeta),(w^{1/2}(w{}x)')(\zeta)\rangle{}_{\mathbb{C}}|\dd\zeta\\
& \leqslant & |(w{}x)(0)|^2+2\,\Bigl(\int_0^{\xi}|(w{}^{1/2}x)(\zeta)|^2\dd\xi\Bigr)^{1/2}\,\Bigl(\int_0^{\xi}|(w{}^{1/2}(w{}x)')(\zeta)|^2\dd\xi\Bigr)^{1/2}\\
& \leqslant & |(w{}x)(0)|^2+2\|x\|_{\Ls^2_{w}(\mathbb{R})}\|(w{}x)'\|_{\Ls^2_{w}(\mathbb{R})}\phantom{\int}
\end{eqnarray*}
for all $\xi\geqslant0$ where we used H\"older's inequality for $p=q=2$. Similarly, we get $|(w{}x)(\eta)|^2\leqslant|(w{}x)(0)|^2+2\|x\|_{\Ls^2_{w}(\mathbb{R})}\|(w{}x)'\|_{\Ls^2_{w}(\mathbb{R})}$ for $\eta\leqslant0$, which shows that $w{}x$ is bounded.

\smallskip

2.~Our estimates above show that $\lim_{\xi\rightarrow\infty}|(wx)^2(\xi)|$ and $\lim_{\xi\rightarrow-\infty}|(wx)^2(\xi)|$ exist. But in view of
$$
\int_{\mathbb{R}}|(wx)(\xi)|^2\dd\xi\leqslant \sup_{\xi\in\mathbb{R}}w(\xi)\int_{\mathbb{R}}w(\xi)|x(\xi)|^2\dd\xi
$$
we have that $wx\in \Ls^2(\mathbb{R})$ holds. Consequently, the two limits above need to be zero.
\end{proof}

\smallskip

The proofs of the two basic generation results that we will prove in Section \ref{SEC:GEN} are based on an explicit formula for the generated semigroup. In order to prove the corresponding results we need the following technical lemma.

\smallskip

\begin{lem}\label{LEM-1} Let $w\colon\mathbb{R}\rightarrow\mathbb{R}$ be a continuous function with $\frac{1}{w}\cdot\mathbb{1}_{(-\infty,0)},\,\frac{1}{w}\cdot\mathbb{1}_{(0,\infty)}\not\in \Ls^1(\mathbb{R})$ and without zeros. We define the two auxiliary functions\vspace{-10pt}
\begin{eqnarray*}
p_w\colon\mathbb{R}\rightarrow\mathbb{R},&&\hspace{-15pt}p_w(\xi):=\int_0^{\xi}w(\zeta)^{-1}\dd\zeta\\
\mu_w\colon\mathbb{R}\times\mathbb{R}\rightarrow\mathbb{R},&&\hspace{-15pt}\mu_w(\xi,t):=p_w^{-1}(p_w(\xi)+t)-\xi
\end{eqnarray*}
where $\mu_w$ is well-defined as $p_w$ is bijective. Let $w$ be positive. Then the following is true.\vspace{4pt}
\begin{compactitem}
\item[(i)] The maps $p_w$ and $p_w^{-1}$ are strictly increasing. We have
$$
p_w|_{(-\infty,0)},\,p_w^{-1}|_{(-\infty,0)}<0,\;\;p_w|_{(0,\infty)},\,p_w^{-1}|_{(0,\infty)}>0,\;\;\lim_{\xi\rightarrow\pm\infty}p_w(\xi)=\pm\infty \; \text{ and } \lim_{\xi\rightarrow\pm\infty}p_w^{-1}(\xi)=\pm\infty.
$$
\item[(ii)] The maps $p_w$ and $p_w^{-1}$ are continuously differentiable with
$$
p_w'(\xi)=1/w(\xi)\;\text{ and } (p^{-1}_w)'(\xi)=w(p_w^{-1}(\xi)).
$$
\item[(iii)] We have $\mu_w(\xi,t)\not=0$ for every $\xi\in\mathbb{R}$ and $t>0$.

\vspace{3pt}

\item[(iv)] We have $\mu_w(\xi,0)=0$ and $\mu_w(0,t)=p^{-1}(t)$ for every $\xi\in\mathbb{R}$ and $t\geqslant0$.

\vspace{3pt}

\item[(v)] We have $t=p(\xi+\mu_w(\xi,t))-p(\xi)$ for every $\xi\in\mathbb{R}$ and $t\geqslant0$.

\vspace{3pt}

\item[(vi)] We have $\mu_w(\xi,t)+\mu_w(\xi+\mu_w(\xi,t),s)=\mu_w(\xi,s+t)$ for $\xi\in\mathbb{R}$ and $s,t\geqslant0$.

\vspace{3pt}

\item[(vii)] We have $\mu_w(\cdot,t)\rightarrow0$ uniformly on compact subsets of $\mathbb{R}$ for $t\searrow0$.

\vspace{3pt}

\item[(viii)] For every $\xi\in\mathbb{R}$ we have ${\displaystyle\lim_{t\searrow0}}\frac{\mu_w(\xi,t)}{t}=w(\xi)$.

\vspace{1pt}

\item[(ix)] The function $\mu_w|_{\mathbb{R}\times[0,\infty)}$ is partially continuously differentiable with
$$
\frac{\partial}{\partial\xi}\,\mu_w(\xi,t) = \frac{w(\xi+\mu_w(\xi,t))}{w(\xi)}-1 \;\text{ and } \frac{\partial}{\partial t}\,\mu_w(\xi,t) = w(\xi+\mu_w(\xi,t)).
$$
\item[(x)] We have $\mu_{-w}(\xi,t)=\mu_w(\xi,-t)$ for all $\xi$, $t\in\mathbb{R}$.
\end{compactitem}
\end{lem}
\begin{proof} In the proof we write $p:=p_w$ and $\mu:=\mu_w$ to simplify the notation.
\smallskip
\\(i)\,--\,(v) This is an easy computation.
\smallskip
\\(vi) Applying (v) three times allows to compute
\begin{eqnarray*}
p(\xi+\mu(\xi,s+t))-p(\xi) & = & s+t\\
& = & p(\xi+\mu(\xi,t)+\mu(\xi+\mu(\xi,t),s)) - p(\xi+\mu(\xi,t)) + p(\xi+\mu(\xi,t))- p(\xi)\\
& = & p(\xi+\mu(\xi,t)+\mu(\xi+\mu(\xi,t),s))- p(\xi) 
\end{eqnarray*}
which yields the desired equality by adding $p(\xi)$ on both sides and using that $p$ is injective.
\smallskip
\\(vii) We have to establish that
$$
\forall\:k>0,\,\epsilon>0\;\exists\:t_0>0\;\forall\: t\in[0,t_0],\,\xi\in[-k,k]\colon |\mu(\xi,t)|<\epsilon
$$
holds. Let $k>0$ and $\epsilon>0$ be given. Since $p^{-1}\colon[p(-k),p(k)+1]\rightarrow\mathbb{R}$ is uniformly continuous we find $\delta>0$ such that $|p^{-1}(\zeta_1)-p^{-1}(\zeta_2)|<\epsilon$ holds for all $\zeta_1,\zeta_2\in[p(-k),p(k)+1]$ with $|\zeta_1-\zeta_2|<\delta$. We put $t_0:=\min\{\delta,1\}$. Let $t\in[0,t_0]$ and $\xi\in[-k,k]$ be given. Then $\zeta_2:=p(\xi)\in[p(-k),p(k)]\subseteq[p(-k),p(k)+1]$ and thus $\zeta_1:=p(\xi)+t\in[p(-k),p(k)+1]$ and we have $|\zeta_1-\zeta_2|=t\leqslant t_0\leqslant\delta$. Therefore
$$
|\mu(\xi,t)|=|p^{-1}(p(\xi)+t)-\xi| = |p^{-1}(p(\xi)+t)-p^{-1}(p(\xi))|=|p^{-1}(\zeta_1)-p^{-1}(\zeta_2)|<\epsilon
$$
is valid as desired.
\smallskip
\\(viii) We fix $\xi\in\mathbb{R}$ and we use (v) and the definition of $p$ to compute
$$
t = p(\xi+\mu(\xi,t))-p(\xi) = \int_0^{\xi+\mu(\xi,t)}w(\zeta)^{-1}\dd\zeta - \int_0^{\xi}w(\zeta)^{-1}\dd\zeta = \int_{\xi}^{\xi+\mu(\xi,t)}w(\zeta)^{-1}\dd\zeta.
$$
For $t>0$ we know by (iii) that $\mu(\xi,t)\not=0$ and we know that $w(\xi)>0$. Therefore, we can use the above to get
$$
\lim_{t\searrow0}\frac{t}{\mu(\xi,t)}=\lim_{t\searrow0}\frac{1}{\mu(\xi,t)}\int_{\xi}^{\xi+\mu(\xi,t)}w(\zeta)^{-1}\dd\zeta=w(\xi)^{-1}
$$
since $\lim_{t\searrow0}\mu(\xi,t)=0$ holds by (vii). The desired statement $\lim_{t\searrow0}\frac{\mu(\xi,t)}{t}=w(\xi)$ follows by taking reciprocals.
\smallskip
\\(ix) Let $t\geqslant0$ be fixed. Using (ii) we compute
\begin{eqnarray*}
\frac{\dd}{\dd\xi}\,\mu(\xi,t) &=& \frac{\dd}{\dd\xi}\,\bigl[p^{-1}(p(\xi)+t)-\xi\bigr]=(p^{-1})'(p(\xi)+t)\cdot\frac{\dd}{\dd\xi}(p(\xi)+t)-1\\
& = & w(p^{-1}(p(\xi)+t))\cdot p'(\xi)-1 = \frac{w(\xi+\mu(\xi,t))}{w(\xi)}-1
\end{eqnarray*}
which is by the above a continuous function. On the other hand, for fixed $\xi\in\mathbb{R}$ we have
\begin{equation*}
\frac{\dd}{\dd t}\,\mu(\xi,t) = \frac{\dd}{\dd t}\,\bigl[p^{-1}(p(\xi)+t)-\xi\bigr] = (p^{-1})'(p(\xi)+t)\cdot1=w(\xi+\mu(\xi,t))
\end{equation*}
which is also a continuous function.
\smallskip
\\(x) Let $\xi$ and $t\in\mathbb{R}$. Put $\eta:=p_w^{-1}(-\xi)$. Then $\xi=-p_w(\eta)=p_{-w}(\eta)$ and thus we established $p_{-w}^{-1}(\xi)=\eta=p_w^{-1}(-\xi)$ for every $\xi\in\mathbb{R}$. Now we can compute
$$
\mu_{-w}(\xi,t) = p_{-w}^{-1}(p_{-w}(\xi)+t)-\xi = p_{w}^{-1}(-p_{-w}(\xi)-t)-\xi= p_{w}^{-1}(p_{w}(\xi)-t)-\xi = \mu_w(\xi,-t)
$$
which finishes the proof.
\end{proof}


\section{Boundary control systems and the weiss theorem}\label{INTERLUDE}

In order to make this article as self-contained as possible we summarize below several notions from systems theory and formulate a version of the Weiss theorem \cite{Weiss} that will turn out crucial for our purposes in the next section. We will follow closely the approach given in \cite[Chapter 11-13]{JZ}.

\smallskip

We point out that it is not feasible to give a detailed introduction to systems theory at this point. For this we refer the reader, e.g., to Tucsnak, Weiss \cite{TW2014} and Staffans \cite{Staffans2005}. For a survey on the transfer function we refer to Zwart \cite{ZwartTrans}.

\smallskip

\begin{dfn}(\cite[Assumption 13.1.2]{JZ})\label{JZ-ASS} We say that
$$
(\Sigma)\;\;\begin{cases}
\;\dot{x}(t)=\mathcal{A}x(t)\\
\;u(t)=\mathcal{B}x(t)\\
\;y(t)=\mathcal{C}x(t)
\end{cases}
$$
is a \emph{boundary control system} if the following holds. \vspace{7pt}

\begin{compactitem}
\item[(i)] $\mathcal{A}\colon D(\mathcal{A})\subseteq X\rightarrow X$, $\mathcal{B}\colon D(\mathcal{B})\subseteq X\rightarrow U$ and $\mathcal{C}\colon D(\mathcal{A})\subseteq X\rightarrow Y$ are linear operators, $D(\mathcal{A})\subseteq D(\mathcal{B})$ holds and $X$, $U$ and $Y$ are Hilbert spaces.\vspace{3pt}

\item[(ii)] The operator $A\colon D(A)\rightarrow X$ with $D(A)=D(\mathcal{A})\cap\operatorname{ker}(\mathcal{B})$ and $Ax=\mathcal{A}x$ for $x\in D(A)$ generates a $\Cnull$-semigroup on $X$.\vspace{3pt}

\item[(iii)] There exists a bounded linear operator $B\colon U\rightarrow X$ such that for all $u\in U$ we have $Bu\in D(\mathcal{A})$, $\mathcal{A}B\colon U\rightarrow X$ is bounded and $\mathcal{B}Bu=u$ holds for $u\in U$.\vspace{3pt}

\item[(iv)] The operator $\mathcal{C}\colon D(A)\rightarrow Y$ is bounded with respect to the graph norm on $D(A)$.\vspace{3pt}
\end{compactitem}
\end{dfn}

\begin{dfn}(\cite[Definition 13.1.3]{JZ})\label{JZ-WP} We say that the boundary control system $(\Sigma)$ is \textit{well-posed} if for every $\tau>0$ there exists $m_{\tau}>0$ such that for every $x_0\in D(\mathcal{A})$ and every $u\in\operatorname{C}^2([0,\tau],U)$ with $u(0)=\mathcal{B}x_0$ the unique classical solution $x$ satisfies 
\begin{equation}\label{DFN-WP}
\|x(\tau)\|_{X}^2+\int_0^{\tau}\|y(t)\|^2\dd t\leqslant m_{\tau}\Big(\|x_0\|_{X}^2+\int_0^{\tau}\|u(t)\|^2\dd t\Big).
\end{equation}
\end{dfn}

In addition to the definition, we state the following very useful test for well-posedness.

\begin{prop}\label{PROP-WP}(\cite[Proposition 13.1.4]{JZ}) Let $(\Sigma)$ be a boundary control system. If every classical solution of the system satisfies
$$
\frac{\operatorname{d}}{\operatorname{d}x}\|x(t)\|^2=\|u(t)\|^2-\|y(t)\|^2
$$
then the system is well-posed.\hfill\qed
\end{prop}

\begin{dfn}(\cite[Theorem 12.1.3]{JZ}) Given a boundary control system, then the \emph{transfer function} $G(s)$ for $s\in\rho(A)$ is given by
$$
G(s)=\mathcal{C}(s-A)^{-1}(\mathcal{A}B-sB)+\mathcal{C}B.
$$
\end{dfn}

In addition to the definition, we state the following very useful way to compute values of the transfer function.

\begin{prop} (\cite[Theorem 12.1.3]{JZ}) \label{JZ-PROP-2} For $s\in\rho(A)$ and $u_0\in U$, $G(s)u_0$ can be computed as the unique solution of 
\begin{equation}\label{DFN-TRANS}
\begin{array}{rcl}
sx_0&=&\mathcal{A}x_0\\
u_0&=&\mathcal{B}x_0\\
G(s)u_0&=&\mathcal{C}x_0
\end{array}
\end{equation}
with $x_0\in D(\mathcal{A})$.\hfill\qed
\end{prop}

\begin{thm} (Weiss \cite{Weiss}, see \cite[Theorem 12.1.3]{JZ})\label{WEISS-THM} Assume that the boundary control system $(\Sigma)$ is well-posed. Let $G$ denote its transfer function. Let $F$ be a bounded linear operator from $Y$ to $U$ and assume that the inverse of $I+G(s)F$ exists and is bounded for $s$ in some right half-plane. Then
$$
(\Sigma')\;\;\begin{cases}
\;\dot{x}(t)=\mathcal{A}x(t)\\
\;u(t)=(\mathcal{B}+F\mathcal{C})x(t)\\
\;y(t)=\mathcal{C}x(t)
\end{cases}
$$
is again a well-posed boundary control system.\hfill\qed
\end{thm}

\smallskip

We mention that one can---by ignoring any deeper meaning of the systems theory notation above---read the Weiss theorem \textquotedblleft{}just\textquotedblright{} as a perturbation theorem for generators of $\Cnull$-semigroups.  The condition that has to be checked then in order to apply the theorem is that the transfer function $G(s)$ is suitably small. This is exactly what we will do below to obtain the result in Proposition \ref{WEISS}.

\medskip

The Weiss theorem however yields more than \textquotedblleft{}just\textquotedblright{} a generator. This additional value requires the language of systems theory to be explained properly and will be explained in Section \ref{SEC:WP} in particular for those readers interested in systems theory.


\section{Generation}\label{SEC:GEN}\smallskip

In this section various generation results will be established. We proceed here according to the outline of the proof of our main results Theorem \ref{PHS-THM} and Corollary \ref{PHS-COR} that we gave in Section \ref{SEC:INTRO} but in opposite order. That is, we start with treating \textquotedblleft{}weighted transport equations\textquotedblright{}. Then we consider an $n$-dimensional diagonal situation, classify the boundary conditions that lead to generators via the Weiss Theorem and then, in the end reduce the general case to the latter via diagonalization. Notice that the assumptions on the Hamiltonian, e.g., smoothness and boundedness, vary throughout this section.

\smallskip

We start by considering the weighted transport equation on a whole axis.

\smallskip

\begin{prop}\label{PROP-1} Let $w\colon\mathbb{R}\rightarrow\mathbb{R}$ be a continuous function with $\frac{1}{w}\cdot\mathbb{1}_{(-\infty,0)},\,\frac{1}{w}\cdot\mathbb{1}_{(0,\infty)}\not\in \Ls^1(\mathbb{R})$ and without zeros. Then the operator $A_{w}\colon D(A_{w})\rightarrow \Ls^2_{|w|}(\mathbb{R})$ given by
\begin{eqnarray*}
A_{w}x &= &(wx)'\\
D(A_{w}) &= &\bigl\{x\in \Ls^2_{|w|}(\mathbb{R})\:;\:wx\in \AC(\mathbb{R})\,\text{ and }\,(wx)'\in \Ls^2_{|w|}(\mathbb{R})\bigr\}
\end{eqnarray*}
generates a unitary $\Cnull$-group $(T_{w}(t))_{t\in\mathbb{R}}$ given by
$$
(T_{w}(t)x)(\xi)=\frac{w(\xi+\mu_{w}(\xi,t))}{w(\xi)}x(\xi+\mu_{w}(\xi,t)) 
$$
for $x\in \Ls_{|w|}^2(\mathbb{R})$ and $\xi\in\mathbb{R}$. Here, $p_w$ and $\mu_w$ denote the maps defined in Lemma \ref{LEM-1}.
\end{prop}

\begin{proof} Let $w$ be positive. We first claim that $T_w(t)\colon \Ls^2_{w}(\mathbb{R})\rightarrow \Ls^2_{w}(\mathbb{R})$ is well-defined and even isometric for fixed $t\geqslant0$. For this we define the map $\nu\colon\mathbb{R}\rightarrow\mathbb{R}$, $\nu(\xi):=\xi+\mu_w(\xi,t)$. Since $\nu(\xi)=p^{-1}_w(p_w(\xi)+t)$ holds, we can use Lemma \ref{LEM-1}(i) to see that
\begin{equation}\label{LIM-NU}
\lim_{\xi\rightarrow\pm\infty}\nu(\xi)=\pm\infty
\end{equation}
is true. From Lemma \ref{LEM-1}(ix) we obtain $\frac{\dd\nu}{\dd\xi}=\frac{w(\nu(\xi))}{w(\xi)}$. By substitution it follows that
$$
\|T_w(t)x\|_{\Ls^2_{w}(\mathbb{R})}^2=\|x\|_{\Ls^2_{w}(\mathbb{R})}^2
$$
holds, which establishes our first claim. Adapting the statements of Lemma \ref{LEM-1}(i)--(ix) for $-w$, it is easy to see that also $T_{-w}(t)\colon\Ls^2_{w}(\mathbb{R})\rightarrow \Ls^2_{w}(\mathbb{R})$ is a well-defined isometry. Using Lemma \ref{LEM-1}(x) it follows that $T_w(-t)=T_{-w}(t)$ holds for all $t\in\mathbb{R}$. Let $t_0>0$ be fixed. A straighforward computation, again employing Lemma \ref{LEM-1}(x), shows that $T_w(t_0)T_{-w}(t_0)x=x$ holds for every $x\in \Ls_w^2(\mathbb{R})$. This shows that $T_w(t_0)$ is invertible with inverse $T_w(t_0)^{-1}=T_{-w}(t_0)$. Employing \cite[Proposition on p.\,80]{EN} it is enough to show that $(T_w(t))_{t\geqslant0}$ satisfies the evolution property and is strongly continuous to conclude that $(T_w(t))_{t\in\mathbb{R}}$ as defined in Proposition \ref{PROP-1} is a $\Cnull$-group. Each operator in this group will then be an isometry: For $t\geqslant0$ we showed this already and for $t<0$ we have $T_w(t)=T_{-w}(-t)$.

\medskip

In order to check the evolution property and strong continuity, we simplify our notation by setting $p:=p_w$, $\mu:=\mu_w$, $A:=A_w$, and $T(t):=T_w(t)$ for $t\geqslant0$. It is easy to see that
$$
(T(0)x)(\xi)=x(\xi) \;\; \text{ and } \;\; (T(t)T(s)x)(\xi)=(T(t+s)x)(\xi)
$$
holds for $x\in \Ls^2_{w}(\mathbb{R})$, $\xi\in\mathbb{R}$ and $t$, $s\geqslant0$, by applying Lemma \ref{LEM-1}(iv) and \ref{LEM-1}(vi).

\medskip

Next we show that $\T$ is strongly continuous. We thus consider only $t\geqslant0$ in the arguments below. We fix $x\in C_c(\mathbb{R})\subseteq \Ls^2_{w}(\mathbb{R})$, we select $k>1$ such that $\supp x\subseteq[-k+1,k-1]$. Using Lemma \ref{LEM-1}(vii) we select $T>0$ such that $|\mu(\xi,t)|\leqslant1$ holds for all $\xi\in[-k,k]$ and $t\in[0,T]$. For $\xi\not\in[-k,k]$ and $t\in[0,T]$ we have $\xi,\,\xi+\mu(\xi,t)\not\in[-k+1,k-1]$ which means $x(\xi)=x(\xi+\mu(\xi,t))=0$. For $t\in[0,T]$ we can thus compute
\begin{eqnarray}\label{EQ-3}
\|T(t)x-x\|_{\infty} & = & \sup_{\xi\in[-k,k]}\,\Bigl|\frac{w(\xi+\mu(\xi,t))}{w(\xi)}x(\xi+\mu(\xi,t))-x(\xi)\Bigr|\nonumber\\
&\leqslant & \sup_{\xi\in[-k,k]}\frac{w(\xi+\mu(\xi,t))}{w(\xi)}\bigl|x(\xi+\mu(\xi,t))-x(\xi)\bigr| + \sup_{\xi\in[-k,k]}\Bigl|\frac{w(\xi+\mu(\xi,t))}{w(\xi)}-1\Bigr|\bigl|x(\xi)\bigr|\nonumber\\
&\leqslant&\sup_{\stackrel{\scriptscriptstyle\xi\in[-k,k]}{\scriptscriptstyle s\in[0,T]}}\frac{w(\xi+\mu(\xi,s))}{w(\xi)}\cdot\sup_{\xi\in[-k,k]}\bigl|x(\xi+\mu(\xi,t))-x(\xi)\bigr|\nonumber\\
& & \phantom{++++++}+ \sup_{\xi\in[-k,k]}\frac{|x(\xi)|}{w(\xi)}\cdot\sup_{\xi\in[-k,k]}\bigl|w(\xi+\mu(\xi,t))-w(\xi)\bigr|.
\end{eqnarray}
Let $y\in\{x,w\}$ be given. We claim that
$$
\forall\:\epsilon>0\;\exists\:T_0\in(0,T]\;\forall\:t\in[0,T_0],\xi\in[-k,k]\colon|y(\xi+\mu(\xi,t))-y(\xi)|<\epsilon
$$
holds. Let $\epsilon>0$ be given. Since $y\colon[-k-1,k+1]\rightarrow\mathbb{C}$ is uniformly continuous we can select $\delta>0$ such that $|y(\xi_1)-y(\xi_2)|<\epsilon$ holds for all $\xi_1,\xi_2\in[-k-1,k+1]$ with $|\xi_1-\xi_2|<\delta$. By Lemma \ref{LEM-1}(vii) there exists $t_0>0$ such that
$$
\forall\:t\in[0,t_0],\,\xi\in[-k,k]\colon|\mu(\xi,t)|<\delta
$$
holds. We put $T_0:=\max\{t_0,T\}$ and consider an arbitrary $t\in[0,T_0]$. We put $\xi_1:=\xi+\mu(\xi,t)$, $\xi_2:=\xi$, which both belong to $[-k-1,k+1]$, since $t\leqslant T$, and satisfy $|\xi_1-\xi_2|=|\mu(\xi,t)|<\delta$. Therefore, $|y(\xi+\mu(\xi,t))-y(\xi)|=|y(\xi_1)-y(\xi_2)|<\epsilon$ holds. We showed that
$$
\lim_{t\rightarrow0}\sup_{\xi\in[-k,k]}\bigl|y(\xi+\mu(\xi,t))-y(\xi)\bigr|=0.
$$
In view of \eqref{EQ-3} this shows $\lim_{t\rightarrow0}\|T(t)x-x\|_{\infty}=0$.

\smallskip

We observe that our selections of $T>0$ and $k>1$ above guarantee that 
$$
\forall\:t\in[0,T]\colon \supp T(t)x\subseteq[-k,k]
$$
holds. Indeed, if we fix $t\in[0,T]$ and take $\xi\in\mathbb{R}$ with $(T(t)x)(\xi)\not=0$ then we have necessarily $x(\xi+\mu(\xi,t))\not=0$ which means $\xi+\mu(\xi,t)\in\supp x\subseteq[-k+1,k-1]$. Since $|\mu(\xi,t)|\leqslant1$ holds, we obtain that $\xi\in[-k,k]$ is valid. By the latter, we can consider
$$
\forall\:t\in[0,T]\colon T(t)x, x\in C[-k,k]
$$
and by the last paragraph we have $\lim_{t\rightarrow0}T(t)x=x$ uniformly on $[-k,k]$. Consequently, $\lim_{t\rightarrow0}T(t)x=x$ holds in particular with respect to the norm $\|\cdot\|_{\Ls^2(-k,k)}$. Since on $\Ls^2(-k,k)$ the latter norm is equivalent to $\|\cdot\|_{\Ls^2_{w}(-k,k)}$, we get
$$
\lim_{t\rightarrow0}\|T(t)x-x\|_{\Ls^2_{w}(\mathbb{R})}=\lim_{t\rightarrow0}\|T(t)x-x\|_{\Ls^2_{w}(-k,k)}=0.
$$
By \cite[Proposition I.5.3]{EN} it follows that $\T$ is strongly continuous.

\medskip

Let $B\colon D(B)\subseteq \Ls^2_{w}(\mathbb{R})\rightarrow \Ls^2_{w}(\mathbb{R})$ be the generator of $\T$. We first claim that $B\subseteq A$ holds. Let $x\in D(B)$ be given and put $y:=Bx\in \Ls^2_{w}(\mathbb{R})$. Then $\lim_{t\rightarrow0}\frac{T(t)x-x}{t}-y=0$ holds and thus
\begin{equation}\label{EQ-3a}
\lim_{t\rightarrow0}\int_a^b\frac{(T(t)x)(\xi)-x(\xi)}{t}\dd\xi=\int_a^by(\xi)\dd\xi
\end{equation}
can be concluded since we have $0<\inf_{\xi\in I}w(\xi)\leqslant\sup_{\xi\in I}w(\xi)<\infty$ for any compact interval $I\subseteq\mathbb{R}$. On the other hand for fixed $t\geqslant0$ we compute
\begin{equation}\label{EQ-4}
\int_a^b\frac{(T(t)x)(\xi)-x(\xi)}{t}\dd\xi = \frac{1}{t}\int_{b}^{b+\mu(b,t)}x(\nu)\dd\nu-\frac{1}{t}\int_a^{a+\mu(b,t)} x(\nu)\,\dd\nu
\end{equation}
by substitution with $\nu$ as at the beginning of this proof. Next we use that $\lim_{t\rightarrow0}\mu(a,t)/t=w(a)$ holds by Lemma \ref{LEM-1}(viii) and that $\lim_{t\rightarrow0}\mu(a,t)=0$ holds by Lemma \ref{LEM-1}(vii) to conclude that
\begin{eqnarray*}
\lim_{t\rightarrow0}\frac{1}{t}\int_{a}^{a+\mu(a,t)}x(\xi)\dd\xi & = & \lim_{t\rightarrow0}\frac{\mu(a,t)}{t}\cdot\frac{1}{\mu(a,t)}\int_{a}^{a+\mu(a,t)}x(\xi)\dd\xi\\
& = & \lim_{t\rightarrow0}\frac{\mu(a,t)}{t}\cdot\lim_{\mu\rightarrow0}\frac{1}{\mu}\int_a^{a+\mu}x(\xi)\dd\xi \; = \; w(a)x(a)
\end{eqnarray*}
is true for almost every $a\in\mathbb{R}$. In the last step we employed that
$$
\lim_{t\rightarrow0}\frac{1}{\mu}\int_a^{a+\mu}x(\xi)\dd\xi=x(a)
$$
holds for almost every $a\in\mathbb{R}$ in view of \cite[Remark after Theorem 9-8\:VI]{Tay85} and by using that $x\in \Ls^2_{w}(\mathbb{R})\subseteq \Ls^1_{\text{loc}}(\mathbb{R})$ holds. We treat the other summand in \eqref{EQ-4} similarly and obtain in view of \eqref{EQ-3a} that
$$
\int_a^by(\xi)\dd\xi=\lim_{t\rightarrow0}\int_a^b\frac{(T(t)x)(\xi)-x(\xi)}{t}\dd\xi= w(b)x(b)-w(a)x(a)
$$
holds for almost all $a,b\in\mathbb{R}$. By changing $x$ on a null set we get that
\begin{equation*}
(w{}x)(b)=(w{}x)(a)+\int_a^by(\xi)\dd\xi
\end{equation*}
holds for all $a,b\in\mathbb{R}$. If we fix $a\in\mathbb{R}$ the latter equation and \cite[Lemma 3.31]{GL} show in particular that $w{}x$ is continuous and that $(w{}x)'=y$ holds almost everywhere. As $y\in \Ls^2_{w}(\mathbb{R})$ holds by definition, we showed $w{}x\in D(A)$ and $Ax=(w{}x)'=y=Bx$ which establishes $B\subseteq A$.

\medskip

It remains to check that $B=A$ holds. As $\T$ is a semigroup of contractions, \cite[Theorem II.1.10(ii)]{EN} implies that $1\in\rho(B)$ holds. In view of \cite[Exercise IV.1.21(5)]{EN} it is enough to show that $1\in\rho(A)$ holds, to conclude $A=B$. We fix $\xi\in\mathbb{R}$ and define $s\colon[0,\infty)\rightarrow\mathbb{R}$, $s(t)=p^{-1}(p(\xi)+t)$. From Lemma \ref{LEM-1}(i) we conclude
\begin{equation*}
\lim_{t\rightarrow\infty}s(t)=\infty
\end{equation*}
and by definition it follows that $s(0)=p^{-1}(p(\xi)+0)=\xi$ is true. We observe that $s(t)=\xi+\mu(\xi,t)$ holds for all $t\geqslant0$ and use Lemma \ref{LEM-1}(ix) to conclude that $s$ is continuously differentiable with
\begin{equation*}
\frac{\dd s}{\dd t}=w(\xi+\mu(\xi,t)).
\end{equation*}
Moreover we can compute $p(s(t))=p(\xi+\mu(\xi,t))=p(p^{-1}(p(\xi)+t))=p(\xi)+t$ which implies that $t=p(s(t))-p(\xi)$ holds for all $t\geqslant0$. We fix $\theta>0$ and $x\in \Ls^2_{w}(\mathbb{R})$. Then $\theta\in\rho(B)$ and we get by substitution
\begin{equation*}
(R(\theta,B)x)(\xi) = \int_{\xi}^{\infty}\e^{-\theta{}(p(s)-p(\xi))}\frac{x(s)}{w(\xi)}\,\dd s
\end{equation*}
which defines the resolvent $R(\theta,B)\in L(\Ls^2_{w}(\mathbb{R}),\Ls^2_{w}(\mathbb{R}))$. We claim that $R(\theta,B)$ is an inverse for the operator $\theta-A\colon D(A)\rightarrow \Ls^2_{w}(\mathbb{R})$. For $x\in D(A)$ and $\xi\in\mathbb{R}$ we compute
\begin{eqnarray*}
[R(\theta,B)(\theta-A)x](\xi)\!\!&=&\!\!\lim_{R\rightarrow\infty}\Bigl(\int_{\xi}^R\!\!\e^{-\theta{}(p(t)-p(\xi))}\frac{\theta{}x(t)}{w(\xi)}\,\dd t-\e^{-\theta(p(R)-p(\xi))}\frac{(w{}x)(R)}{w(\xi)}+x(\xi)\nonumber\\
& &\phantom{++++++++++++++++}-\int_{\xi}^R\frac{\theta{}}{w(t)}\e^{-\theta{}(p(t)-p(\xi))}\frac{w{}(t)x(t)}{w(\xi)}\,\dd t\Bigr)\nonumber\\
&=& \!\!x(\xi)-\frac{\e^{\theta{}p(\xi)}}{w(\xi)}\lim_{R\rightarrow\infty}\e^{-\theta{}p(R)}(w{}x)(R)\;\;=\;\;x(\xi)
\end{eqnarray*}
where the last equality follows from $\lim_{R\rightarrow\infty}\e^{-\theta{}p(R)}=0$ and the fact that $w{}x$ is bounded by Lemma \ref{BAR-LEM}. We know that $\ran R(\theta,B)=D(B)\subseteq D(A)$ holds and thus we can compute
$$
(\theta-A)R(\theta,B)x=(\theta-B)R(\theta,B)x=x
$$
for $x\in \Ls^2_{w}(\mathbb{R})$. This finishes the proof that the operator $A_w$ given in Proposition \ref{PROP-1} is the generator of the $\Cnull$-semigroup $(T_w(t))_{t\geqslant0}$ and thus also of the $\Cnull$-group $(T_w(t))_{t\in\mathbb{R}}$ of consisting of isometries.

\medskip

In order to treat the case of negative $w$, it is enough to show that, for positive $w$, the operator $A_{-w}$ generates $(T_{-w}(t))_{t\in\mathbb{R}}$. This follows however immediately since $A_{-w}=-A_w$ and $T_{-w}(t)=T_{w}(-t)$ holds, cf.~\cite[p.~78]{EN}.
\end{proof}

\smallskip

Now we treat the case of operators on the semi-axis. Below we consider a continuous function $\lambda\colon[0,\infty)\rightarrow(0,\infty)$ with $\frac{1}{\lambda}\not\in \Ls^1(0,\infty)$ and work in the space $\Ls^2_{\lambda}(0,\infty)$. We get the following result.

\smallskip

\begin{prop}\label{PROP-SEMI-PLUS} Let $\lambda\colon[0,\infty)\rightarrow(0,\infty)$ be continuous and  such that $\frac{1}{\lambda}\not\in \Ls^1(0,\infty)$. The operator $A_{\lambda}\colon D(A_{\lambda})\rightarrow \Ls^2_{\lambda}(0,\infty)$ given by
\begin{eqnarray*}
A_{\lambda}x &= &(\lambda x)'\\
D(A_{\lambda}) &= &\bigl\{x\in \Ls^2_{\lambda}(0,\infty)\:;\:\lambda{}x\in \AC[0,\infty) \text{ and } (\lambda{}x)'\in \Ls^2_{\lambda}(0,\infty)\bigr\}
\end{eqnarray*}
generates the $\Cnull$-semigroup $(T_{\lambda}(t))_{t\geqslant0}$ given by
$$
(T_{\lambda}(t)x)(\xi)=\frac{\lambda(\xi+\mu_{\lambda}(\xi,t))}{\lambda(\xi)}x(\xi+\mu_{\lambda}(\xi,t)).
$$
\end{prop}
\begin{proof} We select a strictly positive and continuous function $w\colon\mathbb{R}\rightarrow\mathbb{R}$ such that $\frac{1}{w}|_{(-\infty,0]}\not\in \Ls^1(-\infty,0)$ and $w|_{[0,\infty)}=\lambda$. We put $v:=w|_{(-\infty,0)}$. Then we consider $\Ls^2_{v}(-\infty,0)\subseteq \Ls^2_{w}(\mathbb{R})$ which is a closed subspace. Let $(T_{w}(t))_{t\geqslant0}$ be as in Proposition \ref{PROP-1}. We claim that $\Ls^2_{v}(-\infty,0)$ is $(T_{w}(t))_{t\geqslant0}$-invariant. Let $x\in \Ls^2_{v}(-\infty,0)$ be given. We consider $x$ as an element of $\Ls^2_{w}(\mathbb{R})$, that is we have $x|_{(0,\infty)}\equiv0$. For $\xi\geqslant0$ and $t\geqslant0$ we have $\xi\geqslant0\geqslant p_{w}^{-1}(-t)$ which implies $p_{w}(\xi)\geqslant-t$ and thus $p_{w}(\xi)+t\geqslant0$ as $p_w$ is increasing. As $p^{-1}_{w}$ is also increasing, we get
$$ 
\xi+\mu_{w}(\xi,t)=p_{w}^{-1}(p_{w}(\xi)+t)\geqslant p^{-1}_{w}(0)=0
$$
and thus $x(\xi+\mu_{w}(\xi,t))=0$. Recalling the formula of the semigroup we see that this implies $(T_{w}(t)x)(\xi)=0$ which means $T_w(t)x\in \Ls^2_{v}(-\infty,0)$. We consider the quotient map
$$
q\colon \Ls^2_w(\mathbb{R})\longrightarrow\,\quot{\Ls_w^2(\mathbb{R})}{\Ls_v^2(-\infty,0)}
$$
and apply \cite[I.5.13]{EN} to get a $\Cnull$-semigroup $(T(t)_{/})_{t\geqslant0}$ on $\Ls_w^2(\mathbb{R})/\Ls_v^2(-\infty,0)$ given by $T(t)_{/}q(x)=q(T_w(t)x)$ which is generated by $A_{/}\colon D(A_{/})\rightarrow \Ls_w^2(\mathbb{R})/\Ls_v^2(-\infty,0)$ with $D(A_{/})=q(D(A_w))$ and $A_{/}q(x)=q(A_wx)$, see \cite[II.2.4]{EN}. Now we identify $\Ls_w^2(\mathbb{R})/\Ls_v^2(-\infty,0)\cong \Ls^2_{\lambda}(0,\infty)$ via $[x]\mapsto x|_{[0,\infty)}$ which transforms $q$ into the mapping
$$
q\colon \Ls^2_w(\mathbb{R})\longrightarrow \Ls^2_{\lambda}(0,\infty),\;\;q(x)=x|_{[0,\infty)}.
$$
For $x\in \Ls^2_{\lambda}(0,\infty)$ we select $y\in \Ls^2_w(\mathbb{R})$ such that $y|_{[0,\infty)}=x$ holds. Then we have
$$
(T(t)_{/}x)(\xi)=(T(t)_{/}q(y))(\xi)=q(T_w(t)y)(\xi)=\frac{w(\cdot+\mu_{w}(\cdot,t))}{w(\cdot)}y(\cdot+\mu_{w}(\cdot,t))\Bigr|_{[0,\infty)}(\xi)
$$
for $\xi\in[0,\infty)$. But for $\xi\geqslant0$ and $t\geqslant0$ we see that $p_w(\xi)=p_{\lambda}(\xi)$ and thus
$$
\mu_w(\xi,t)=p_w^{-1}(p_w(\xi)+t)-\xi=p_{\lambda}^{-1}(p_{\lambda}(\xi)+t)-\xi=\mu_{\lambda}(\xi,t)
$$
holds. Moreover, we have $\xi+\mu_w(\xi,t)=p_w^{-1}(p_w(\xi)+t)\geqslant0$. Consequently, it follows that
$$
(T(t)_{/}x)(\xi)=\frac{\lambda(\xi+\mu_{\lambda}(\xi,t))}{\lambda(\xi)}x(\xi+\mu_{\lambda}(\xi,t))=(T_{\lambda}(t)x)(\xi)
$$
is true for all $x\in \Ls^2_{\lambda}(0,\infty)$ and $\xi\in[0,\infty)$. Next we compute
\begin{eqnarray*}
D(A_{/})\;=\;q(D(A_w)) & = & \bigl\{y|_{[0,\infty)}\:;\:x\in \Ls^2_{w}(\mathbb{R}),\;wy\in \AC(\mathbb{R}) \text{ and } (wy)'\in \Ls^2_{w}(\mathbb{R})\bigr\}\\
& = & \bigl\{x\in \Ls^2_{\lambda}(0,\infty)\:;\:\lambda{}x\in \AC[0,\infty) \text{ and } (\lambda{}x)'\in \Ls^2_{\lambda}(0,\infty)\bigr\}\\
& = & D(A_{\lambda})
\end{eqnarray*}
where we used that $\lambda{}x\in \AC[0,\infty)$ implies that we can extend $x$ to the whole axis for instance in a way that the extension $y$ belongs to $C^{\infty}(-\infty,0)$ and satisfies $y|_{(-\infty,-1]}\equiv0$. Finally for $x\in \Ls^2_{\lambda}(0,\infty)$ we select $y\in \Ls^2_{w}(0,\infty)$ with $y|_{[0,\infty)}=x$ and compute
$$
(A_{/}x)(\xi)=(A_{/}q(y))(\xi)=q(A_wy)(\xi)=\frac{\partial}{\partial\xi}(-wy)\Bigl|_{[0,\infty)}(\xi)=\frac{\partial}{\partial\xi}(-\lambda{}x)(\xi)=(A_{\lambda}x)(\xi)
$$
for $\xi\in[0,\infty)$. This finishes the proof.
\end{proof}

\smallskip

Next we treat the case of a negative sign, that is we consider a strictly negative and continuous function $\uptheta\colon[0,\infty)\rightarrow(-\infty,0)$ and work in the space $\Ls^2_{|\uptheta|}(0,\infty)$. When we use our results from the whole axis we will stick to the notation that we used most of the time until now, i.e., $w\colon\mathbb{R}\rightarrow(0,\infty)$ will denote a strictly positive function and we consider $A_{-w}$ and $(T_{-w}(t))_{t\in\mathbb{R}}$. We emphasize that the results of Lemma \ref{LEM-1} have to be updated when $p_{-w}$ and $\mu_{-w}$ are used. The auxiliary function $p_{-w}$ is now for instance decreasing, whereas $p_w$ was increasing. In the formula given in Proposition \ref{PROP-1} we can however simply replace $w$ with $-w$ and get the (semi)group generated by $A_{-w}x=-(wx)'$.

\smallskip

\begin{prop}\label{PROP-SEMI-MINUS} Let $\uptheta\colon[0,\infty)\rightarrow(-\infty,0)$ be continuous and  such that $\frac{1}{\uptheta}\not\in \Ls^1(0,\infty)$. The operator $A_{\uptheta}\colon D(A_{\uptheta})\rightarrow \Ls^2_{|\uptheta|}(0,\infty)$ given by
\begin{eqnarray*}
A_{\uptheta}x &= &(\uptheta x)'\\
D(A_{\uptheta}) &= &\bigl\{x\in \Ls^2_{|\uptheta|}(0,\infty)\:;\:\uptheta{}x\in \AC[0,\infty), \; (\uptheta{}x)'\in \Ls^2_{|\uptheta|}(0,\infty) \text{ and } (\uptheta{}x)(0)=0\bigr\}
\end{eqnarray*}
generates the $\Cnull$-semigroup $(T_{\uptheta}(t))_{t\geqslant0}$ given by
$$
(T_{\uptheta}(t)x)(\xi)=\begin{cases}
\;{\displaystyle\frac{\uptheta(\xi+\mu_{\uptheta}(\xi,t))}{\uptheta(\xi)}x(\xi+\mu_{\uptheta}(\xi,t))} & \text{ if } \xi+\mu_{\uptheta}(\xi,t)\geqslant0,\\
\;\hspace{65pt}0 & \text{ otherwise.}
\end{cases}
$$
Here, $\mu_{\uptheta}$ is defined as in Lemma \ref{LEM-1} but $\lambda$ replaced with $\uptheta$.
\end{prop}
\begin{proof} We select a strictly positive and continuous function $w\colon\mathbb{R}\rightarrow(0,\infty)$ such that $\frac{1}{w}|_{(-\infty,0]}\not\in \Ls^1(-\infty,0)$ and $w|_{[0,\infty)}=-\uptheta$. Then we consider $\Ls^2_{|\uptheta|}(0,\infty)\subseteq \Ls^2_{w}(\mathbb{R})$ which is a closed subspace. Let $(T_{-w}(t))_{t\geqslant0}$ be the  $\Cnull$-semigroup from Proposition \ref{PROP-1}. We claim that $\Ls^2_{|\uptheta|}(0,\infty)$ is $(T_{-w}(t))_{t\geqslant0}$-invariant. Let $x\in \Ls^2_{|\uptheta|}(0,\infty)$ be given. We consider $x$ as an element of $\Ls^2_{w}(\mathbb{R})$ that is we have $x|_{(-\infty,0)}\equiv0$. For $\xi\leqslant0$ and $t\geqslant0$ we have $\xi\leqslant0\leqslant p_{-w}^{-1}(-t)$ which implies $p_{-w}(\xi)\geqslant-t$ and thus $p_{-w}(\xi)+t\geqslant0$ as $p_{-w}$ is decreasing. As $p^{-1}_{-w}$ is also decreasing, we get
$$ 
\xi+\mu_{-w}(\xi,t)=p_{-w}^{-1}(p_{-w}(\xi)+t)\leqslant p^{-1}_{-w}(0)=0
$$
and thus $x(\xi+\mu_{-w}(\xi,t))=0$. Recalling the formula of the semigroup we see that this implies $(T_{-w}(t)x)(\xi)=0$ which means $T_{-w}(t)x\in \Ls^2_{|\uptheta|}(0,\infty)$. By \cite[I.5.12]{EN}, $(T_{-w}(t)_|)_{t\geqslant0}$ generates a $\Cnull$-semigroup on $\Ls^2_{|\uptheta|}(0,\infty)$ whose generator is the part of $A_{-w}\colon D(A_{-w})\rightarrow \Ls^2_{w}(\mathbb{R})$ in $\Ls^2_{|\uptheta|}(0,\infty)$, i.e., $A_{-w|}x=A_{-w}x$ for
\begin{eqnarray*}
x\in D(A_{-w|}) & = & \bigl\{x\in D(A_{-w})\cap \Ls^2_{|\uptheta|}(0,\infty)\:;\:A_{-w}x\in \Ls^2_{|\uptheta|}(0,\infty)\bigr\}\\
& = & \bigl\{x \in \Ls^2_{|\uptheta|}(0,\infty)\:;\: wx\in \AC(\mathbb{R}),\; (wx)'\in \Ls^2_w(\mathbb{R}) \text{ and } (-wx)'\in \Ls^2_{|\uptheta|}(0,\infty)\bigr\}\\
& = & \bigl\{x \in \Ls^2_{|\uptheta|}(0,\infty)\:;\: \uptheta{}x\in \AC[0,\infty), \; (\uptheta{}x)'\in \Ls^2_{|\uptheta|}(0,\infty) \text{ and } (\uptheta{}x)(0)=0\bigr\}\\
& = & D(A_{\uptheta})
\end{eqnarray*}
where we used that $x\in \Ls^2_{|\uptheta|}(0,\infty)\subseteq \Ls^2_{w}(0,\infty)$ means $x|_{(-\infty,0)}$ is zero, that $wx\in \AC(\mathbb{R})$ and $\uptheta{}x\in \AC[0,\infty)$ are in particular continuous, and that on $[0,\infty)$ we have $-w=\uptheta$.

\smallskip

For $\xi\geqslant0$ we have $p_{-w}(\xi)=p_{\uptheta}(\xi)$ and for $\eta\leqslant0$ we have $p^{-1}_{-w}(\eta)=p^{-1}_{\uptheta}(\eta)$. Let now $\xi\geqslant0$ and $t\geqslant0$ be such that $p^{-1}_{-w}(p_{-w}(\xi)+t)=\xi+\mu_{-w}(\xi,t)\geqslant0$. Since $p_{-w}^{-1}$ is decreasing, this means $p_{\uptheta}(\xi)+t=p_{-w}(\xi)+t\leqslant0$ and consequently we proved that
$$
\xi+\mu_{\uptheta}(\xi,t)=p_{\uptheta}^{-1}(p_{\uptheta}(\xi)+t)=p_{-w}^{-1}(p_{-w}(\xi)+t)=\xi+\mu_{-w}(\xi,t)
$$
is valid for all $\xi\geqslant0$ and $t\geqslant0$ such that $\mu_{-w}(\xi,t)\geqslant0$. We thus get
\begin{eqnarray*}
(T_{-w}(t)_|x)(\xi) & = & \frac{w(\xi+\mu_{-w}(\xi,t))}{w(\xi)}x(\xi+\mu_{-w}(\xi,t))\\
& = & \left\{\begin{array}{ll}
{\displaystyle\frac{\uptheta(\xi+\mu_{\uptheta}(\xi,t))}{\uptheta(\xi)}x(\xi+\mu_{\uptheta}(\xi,t))} & \text{ if } \xi+\mu_{\uptheta}(\xi,t)\geqslant0\\
\hspace{65pt}0 & \text{ otherwise}        \end{array}\right\} = (T_{\uptheta}(t)x)(\xi) 
\end{eqnarray*}
for each $x\in \Ls^2_{|\uptheta|}(0,\infty)$. This finishes the proof.
\end{proof}

The following lemma will enable us to make the results of Proposition \ref{PROP-SEMI-PLUS} and Proposition \ref{PROP-SEMI-MINUS} vector-valued and to combine the two generators in one operator that multiplies in some coordinates with a positive and in others with a negative function. The proof is straightforward.

\smallskip

\begin{lem}\label{SIMPLE-LEM} For $k=1,\dots,n$ let $A_k\colon D(A_k)\rightarrow H_k$ be generators of $\Cnull$-semigroups $(T_k(t))_{t\geqslant0}$ on Banach spaces $H_k$. Then, $A:=\diag(A_1,\dots,A_n)\colon D(A):=D(A_1)\oplus\cdots\oplus D(A_n)\rightarrow H_1\oplus\cdots\oplus H_n=:H$ generates the $\Cnull$-semigroup $\T$ given by $T(t):=\diag(T_1(t),\dots,T_n(t))$ for $t\geqslant0$ on $H$.\hfill\qed
\end{lem}

\smallskip

Applying Lemma \ref{SIMPLE-LEM} to the generation results that we established so far gives immediately the following.

\smallskip

\begin{prop}\label{PROP-DIAG} Let $n_+,n_-\geqslant1$ be integers and let $n:=n_++n_-$. Let $\Delta\colon[0,\infty)\rightarrow\mathbb{R}^{n\times n}$ be given by $\Delta=\diag(\Lambda,\Theta)$ where $\Lambda\colon[0,\infty)\rightarrow\mathbb{R}^{n_+\times n_+}$ and $\Theta\colon[0,\infty)\rightarrow\mathbb{R}^{n_-\times n_-}$ are given by $\Lambda=\diag(\lambda_1\dots,\lambda_{n_+})$, $\Theta=\diag(\uptheta_1,\dots,\uptheta_{n_-})$ where $\lambda_1,\dots,\lambda_{n_+}\colon[0,\infty)\rightarrow(0,\infty)$ and $\uptheta_1,\dots,\uptheta_{n_-}\colon[0,\infty)\rightarrow(-\infty,0)$ are continuous and $\frac{1}{\lambda_k},\frac{1}{\uptheta_j}\not\in \Ls^1(0,\infty)$ holds for $k=1,\dots,n_+$ and $j=1,\dots,n_-$. Then the operator $A_{\Delta}\colon D(A_{\Delta})\rightarrow \Ls^2_{|\Delta|}(0,\infty)$ given by
\begin{eqnarray*}
A_{\Delta}x &= &(\Delta{}x)'\\
D(A_{\Delta}) &= &\bigl\{x\in \Ls^2_{|\Delta|}(0,\infty)\:;\:\Delta x\in \AC[0,\infty),\;\Delta{}x\in \Ls_{|\Delta|}^2(0,\infty)\text{ and }(\Theta{}x)(0)=0\bigr\}
\end{eqnarray*}
generates a $\Cnull$-semigroup on 
$$
\Ls^2_{|\Delta|}(0,\infty)=\bigl\{x\colon[0,\infty)\rightarrow\mathbb{C}^n\:;\: x \text{ measurable and } \|x\|_{\Ls^2_{|\Delta|}(0,\infty)}^2:=\int_0^{\infty}x(\xi)^{\star}\!|\Delta|x(\xi)\dd\xi<\infty\bigr\}.
$$
with $|\Delta|=\diag(\lambda_1,\dots,\lambda_{n_+},|\uptheta_1|,\dots,|\uptheta_{n_-}|)$.
\end{prop}
\begin{proof} It is enough to apply Lemma \ref{SIMPLE-LEM} to the operators
$$
A_{\lambda_k}\colon D(A_{\lambda_k})\rightarrow \Ls^2_{\lambda_k}(0,\infty)\,\text{ and }\,A_{\uptheta_j}\colon D(A_{-\uptheta_j})\rightarrow \Ls^2_{|\uptheta_j|}(0,\infty)
$$
for $k=1,\dots,n_+$ and $j=1,\dots,n_-$. The latter are generators according to Proposition \ref{PROP-SEMI-PLUS} and Proposition \ref{PROP-SEMI-MINUS} if we use the corresponding domains, i.e., $D(A_{\lambda_k})$ includes no boundary condition and $D(A_{\uptheta_j})$ includes the boundary condition $(\uptheta_jx)(0)=0$. The result follows since
$$
\Ls^2_{|\Delta|}(0,\infty)=\Bigosum{k=1}{n_+}\Ls^2_{\lambda_k}(0,\infty)\oplus\Bigosum{j=1}{n_-}\Ls^2_{\uptheta_j}(0,\infty)\;\text{ and }\;D(A_{\Delta})=\Bigosum{k=1}{n_+}D(A_{\lambda_k})\oplus\Bigosum{j=1}{n_-}D(A_{\uptheta_j})
$$
hold.
\end{proof}

\smallskip

Our next aim is to classify all linear boundary conditions for which $A_{\Delta}$ as in Proposition \ref{PROP-DIAG} is a generator. In order to achieve this, we make use of the terminology of boundary control systems that we reviewed for this purpose in Section \ref{INTERLUDE}. Firstly, we put the operator of Proposition \ref{PROP-DIAG} in the new context and establish in Lemma \ref{BCS-LEM} that $A_{\Delta}$, together with suitable input and output, gives rise to a boundary control system. Secondly, we show that this system is well-posed and compute its transfer function in Lemma \ref{LEM-BCS-WELLPOSED}. Then we are able to apply the Weiss theorem to obtain the generation result of Proposition \ref{WEISS}.

\smallskip

\begin{lem}\label{BCS-LEM} In the situation of Proposition \ref{PROP-DIAG} let $\mathcal{A}\colon D(\mathcal{A})\rightarrow \Ls^2_{|\Delta|}(0,\infty)$, $\mathcal{B}\colon D(\mathcal{B})\rightarrow\mathbb{C}^{n_-}$ and $\mathcal{C}\colon D(\mathcal{C})\rightarrow\mathbb{C}^{n_+}$ be given by
$$
\mathcal{A}x=(\Delta{}x)',\;\mathcal{B}x=(\Theta{}x_-)(0) \,\text{ and }\,\mathcal{C}x=(\Lambda{}x_+)(0)
$$
for $x\in D(\mathcal{A})=D(\mathcal{B})=D(\mathcal{C})=\{x\in \Ls^2_{|\Delta|}(0,\infty)\:;\:\mathcal{H}x\in \AC[0,\infty) \text { and } (\Delta x)'\in \Ls^2_{|\Delta|}(0,\infty)\}$. Then,
$$
(\Sigma)\;\;\begin{cases}
\;\dot{x}(t)=\mathcal{A}x(t)\\
\;u(t)=\mathcal{B}x(t)\\
\;y(t)=\mathcal{C}x(t)
\end{cases}
$$
is a boundary control system.
\end{lem}
\begin{proof} By Proposition \ref{PROP-DIAG}, $\mathcal{A}|_{\operatorname{ker}\mathcal{B}}\colon D(\mathcal{A})\cap\operatorname{ker}\mathcal{B}\rightarrow \Ls^2_{|\Delta|}(0,\infty)$ generates a $\Cnull$-semigroup. We select $\varphi\in C^{\infty}[0,\infty)$ with compact support and such that $\varphi(0)=1$. Then we define
$$
B \colon \mathbb{C}^{n_-}\rightarrow \Ls^2_{|\Delta|}(0,\infty),\;\;(Bu)(\xi)=\bigl[0\;\cdots\;0\;\;\frac{\varphi(\xi)}{\uptheta_1(\xi)}u_1\;\cdots\;\frac{\varphi(\xi)}{\uptheta_{n_-}(\xi)}u_{n_-}\bigr]^T
$$
which is a linear operator with values in $D(\mathcal{A})$ that satisfies
$$
\mathcal{B}Bu=(\Theta(Bu)_-)(0)=\bigl[\uptheta_1(0)\frac{\varphi(0)}{\uptheta_1(0)}u_1\;\cdots\;\uptheta_1(0)\frac{\varphi(0)}{\uptheta_1(0)}u_{n_-}\bigr]^T=[u_1\;\cdots\;u_{n_-}]=u
$$
for each $u\in\mathbb{C}^{n-}$. Moreover, $\mathcal{A}B\colon\mathbb{C}^{n_-}\rightarrow \Ls^2_{|\Delta|}(0,\infty)$ and $\mathcal{C}\colon D(\mathcal{C})\rightarrow\mathcal{C}^{n_-}$ are linear and continuous.
\end{proof}

\smallskip

In the lemma below we establish that the boundary control system from Lemma \ref{BCS-LEM} is well-posed and has zero transfer function. As we explained in Section \ref{INTERLUDE} this is crucial for the application of Theorem \ref{WEISS-THM} which then finally will give us the classification of all linear boundary conditions that make $A_{\Delta}$ a generator. For the result we need the additional assumption that $\Delta$ is bounded. Notice that $\frac{1}{\lambda_k}$, $\frac{1}{\theta_j}\not\in \Ls^1(0,\infty)$ follows automatically.

\smallskip

\begin{lem}\label{LEM-BCS-WELLPOSED} In the situation of Proposition \ref{PROP-DIAG} assume that $\Delta\colon[0,\infty)\rightarrow\mathbb{R}^{n\times n}$ is bounded. Then the boundary control system considered in Lemma \ref{BCS-LEM} is well-posed and its transfer function is zero.
\end{lem}

\begin{proof} 1.~Let $x$ be a classical solution of the boundary control system. We write $x=[x_+\;\;x_-]^T$ and use this notation also later in this proof. The second part of Lemma \ref{BAR-LEM} enables us to compute
\begin{eqnarray*}
\frac{\dd}{\dd t}\|x(t)\|^2_{\Ls^2_{|\Delta|}(0,\infty)} &=& \lim_{R\rightarrow\infty}\Bigl(\frac{\partial}{\partial t}\int_0^R x_+(\xi,t)^{\star}\Lambda(\xi)x_+(\xi,t)\dd\xi + \frac{\partial}{\partial t}\int_0^R x_-(\xi,t)^{\star}(-\Theta(\xi))x_-(\xi,t)\dd\xi\Bigr)\\
& = & \lim_{R\rightarrow\infty}\Bigl((\Lambda(\xi)x_+(\xi,t))^{\star}\Lambda(\xi)x_+(\xi,t)\Bigr|_0^R-(\Theta(\xi)x_-(\xi,t))^{\star}\Theta(\xi)x_-(\xi,t)\Bigr|_0^R\Bigr)\\
& = & \lim_{R\rightarrow\infty}\Bigl( \|\Lambda(R)x_+(R,t)\|_{\mathbb{C}^{n_+}}^2-\|\Lambda(0)x_+(0,t)\|^2_{\mathbb{C}^{n_+}}\phantom{\int}\\
& & \phantom{+++++++}-\|\Theta(R)x_-(R,t)\|^2_{\mathbb{C}^{n_-}}+\|\Theta(0)x_-(0,t)\|^2_{\mathbb{C}^{n_-}}\Bigr)\\
& = & \|\Theta(0)x_-(0,t)\|^2_{\mathbb{C}^{n_-}}-\|\Lambda(0)x_+(0,t)\|^2_{\mathbb{C}^{n_+}}\;\;=\;\; \|u(t)\|^2_{\mathbb{C}^{n_-}}-\|y(t)\|^2_{\mathbb{C}^{n_+}}\phantom{\int}
\end{eqnarray*}
which yields the well-posedness by \cite[Proposition 13.1.4]{JZ}.

\smallskip

2.~The equation in \eqref{DFN-TRANS} read in our situation as follows
\begin{equation}\label{EQ-TRANS}
sx_0=\mathcal{A}x_0=\begin{bmatrix}(\Lambda x_+^{(0)})'\\(\Theta x_-^{(0)})'\end{bmatrix},\;\;u_0=\mathcal{B}x_0=\Theta x_-^{(0)}(0),\;\;G(s)u_0=\mathcal{C}x_0=(\Lambda x_+^{(0)})(0)
\end{equation}
where $x_0=[x_+^{(0)}\;\;\;x_-^{(0)}]^{T}\in D(\mathcal{A})$. The general solution of the first equation in \eqref{EQ-TRANS} is
$$
x_0(\xi)=\begin{bmatrix}\alpha\cdot\Lambda(\xi)^{-1}e^{s\int_0^{\xi}\Lambda(\zeta)^{-1}\dd\zeta}\\\beta\cdot\Theta(\xi)^{-1}e^{s\int_0^{\xi}\Theta(\zeta)^{-1}\dd\zeta}\end{bmatrix}.
$$
Remembering $\Lambda=\diag(\lambda_1,\dots,\lambda_{n_+})$, we see that
$$
\Lambda(\xi)^{-1}e^{s\int_0^{\xi}\Lambda(\zeta)^{-1}\dd\zeta}=\begin{bmatrix}\lambda_1(\xi)^{-1}e^{s\int_0^{\xi}\lambda_1(\zeta)^{-1}\dd\zeta}\vspace{-5pt}\\\vdots\\\lambda_{n_+}(\xi)^{-1}e^{s\int_0^{\xi}\lambda_{n_+}(\zeta)^{-1}\dd\zeta}\end{bmatrix}
$$
holds. For every $k=1,\dots,n_+$ and $\Re s>0$ we have
$$
\lim_{\xi\rightarrow\infty}\lambda_k(\xi)\bigl|\lambda_k(\xi)^{-1}e^{s\int_0^{\xi}\lambda_1(\zeta)^{-1}\dd\zeta}\bigr|= \lim_{\xi\rightarrow\infty}e^{\Re s\int_0^{\xi}\lambda_k(\zeta)^{-1}\dd\zeta}=\infty
$$
as $\frac{1}{\lambda_k}\not\in \Ls^1(0,\infty)$ holds by assumption. As $x_0\in \Ls^2_{|\Delta|}(0,\infty)$ holds, we have necessarily $\alpha=0$ in view of Lemma \ref{BAR-LEM}. On the other hand, we claim that
$$
\|x_-^{(0)}\|_{\Ls^2_{\Lambda}(0,\infty)}^2=\int_0^{\infty}|\Theta(\xi)|\bigl|\Theta(\xi)^{-1} e^{s\cdot\int_0^{\xi}\Theta(\zeta)^{-1}\dd\zeta}\bigr|^2\dd\xi=\sum_{k=1}^{n_-}\int_0^{\infty}\bigl|\uptheta_k(\xi)^{-1}\bigr|\bigl|e^{s\int_0^{\xi}\uptheta_k(\zeta)^{-1}\dd\zeta}\bigr|^2\dd\xi<\infty
$$
holds. For $j=1,\dots,n_-$ we compute
\begin{eqnarray*}
\int_0^{\infty}\bigl|\uptheta_k(\xi)^{-1}\bigr|\bigl|e^{s\int_0^{\xi}\uptheta_j(\zeta)^{-1}\dd\zeta}\bigr|^2\dd\xi & = & -\lim_{R\rightarrow\infty}\int_0^R\uptheta_k(\xi)^{-1}e^{2\Re s\int_0^{\xi}\uptheta_j(\zeta)^{-1}\dd\zeta}\dd\xi\\
& = & -\lim_{R\rightarrow\infty}\frac{1}{2\Re s}e^{2\Re s \int_0^{\xi}\uptheta_j(\zeta)^{-1}\dd\zeta}\Bigr|_{\xi=0}^{\xi=R}\\
& = & \frac{1}{2\Re s}\bigl(1-\lim_{R\rightarrow\infty}e^{2\Re s \int_0^{R}\uptheta_j(\zeta)^{-1}\dd\zeta}\bigr)\;\;=\;\;\frac{1}{2\Re s}
\end{eqnarray*}
since $\frac{1}{\uptheta_j}\not\in \Ls^1(0,\infty)$ holds by assumption. We thus have that
$$
x_0(\xi)=\begin{bmatrix}x_+^{(0)}(\xi)\\x_-^{(0)}(\xi)\end{bmatrix}=\begin{bmatrix}0\\\beta\cdot\theta(\xi)^{-1}e^{s\int_0^{\xi}\Theta^{-1}\dd\zeta}\end{bmatrix}
$$
for $\xi\in[0,\infty)$ defines a function $x_0\in D(\mathcal{A})$ satisfying the first equation in \eqref{EQ-TRANS}. The second equation in \eqref{EQ-TRANS} then yields
$$
u_0=\Theta(0)x_-^{(0)}=\Theta(0)\beta\Theta(0)^{-1}e^{s\cdot0}=\beta
$$
and the third equation in \eqref{EQ-TRANS} leads to
$$
G(s)u_0=\Lambda(0)x_+^{(0)}(0)=0
$$
which shows that $G(s)=0$ whenever $\Re s>0$ holds.
\end{proof}

\smallskip

Now we are ready to classify all boundary conditions that turn $A_{\Delta}$ into a generator on $\Ls^2_{|\Delta|}(0,\infty)$. We point out that is essential that we work in $\Ls^2_{|\Delta|}(0,\infty)$ since $|\Delta|$ is a priori not boundedly invertible and thus $\Ls^2_{|\Delta|}(0,\infty)\not=\Ls^2(0,\infty)$.

\smallskip

\begin{prop}\label{WEISS} Let $\Delta=\diag(\Lambda,\Theta)\colon[0,\infty)\rightarrow\mathbb{R}^{n\times n}$ be as in Lemma \ref{LEM-BCS-WELLPOSED}. Let $K\in\mathbb{C}^{n_-\times{}n_-}$ and $Q\in\mathbb{C}^{n_-\times{}n_+}$ be matrices such that $[K\;\;Q]\in\mathbb{C}^{n_-\times{}n}$ has rank $n_-$. The operator $A_{\Delta}\colon D(A_{\Delta})\rightarrow \Ls^2_{|\Delta|}(0,\infty)$ given by
\begin{eqnarray*}
A_{\Delta}x &= &(\Delta{}x)'\\
D(A_{\Delta}) &= &\bigl\{x\in \Ls^2_{|\Delta|}(0,\infty)\:;\:\Delta{}x\in \AC[0,\infty), \; (\Delta{}x)'\in \Ls^2_{|\Delta|}(0,\infty)\\
& & \phantom{+++++++++++}\text{ and }\;K\Theta(0)x_-(0)+Q\Lambda(0)x_+(0)=0\bigr\}
\end{eqnarray*}
generates a $\Cnull$-semigroup if and only if $K$ is invertible. Here we use the notation $x=[x_+\;\;x_-]^T$.
\end{prop}
\begin{proof}1.~We consider the boundary control system $(\Sigma)$ from Lemma \ref{BCS-LEM}, which is well-posed and whose transfer function is constant zero. We consider the feedback operator $F\colon\mathbb{C}^{n_+}\rightarrow\mathbb{C}^{n_-}$, $Fy=K^{-1}Qy$. Applying Theorem \ref{WEISS-THM} we obtain that\label{SIGMA-PRIME}
$$
(\Sigma')\;\;\begin{cases}
\;\dot{x}(t)=\mathcal{A}x(t)\\
\;u(t)=(\mathcal{B}+F\mathcal{C})x(t)\\
\;y(t)=\mathcal{C}x(t)
\end{cases}
$$
is a boundary control system which is well-posed. In view of Definition \ref{JZ-ASS}(ii) we have in particular that
$$
A\colon D(A)\rightarrow \Ls^2_{|\Delta|}(0,\infty),\;Ax=\mathcal{A}x
$$
generates a $\Cnull$-semigroup with 
\begin{eqnarray*}
D(A) & = & D(\mathcal{A})\cap\ker(\mathcal{B}+F\mathcal{C})\\
 & = &\bigl\{x\in \Ls^2_{|\Delta|}(0,\infty)\:;\:\mathcal{H}x\in \AC[0,\infty),\; (\mathcal{H}x)'\in \Ls^2_{|\Delta|}(0,\infty) \text{ and }(\mathcal{B}+F\mathcal{C})x(t)=0\bigr\}\\
& = &\bigl\{x\in \Ls^2_{|\Delta|}(0,\infty)\:;\:\mathcal{H}x\in \AC[0,\infty),\; (\mathcal{H}x)'\in \Ls^2_{|\Delta|}(0,\infty)\\
& &\phantom{++++++++++++}\text{ and }\:\Theta(0)x_-(0)+K^{-1}Q\Lambda(0)x_+(0)=0\bigr\}\\
& = & D(A_{\Delta})
\end{eqnarray*}
in view of $(\mathcal{B}+F\mathcal{C})x(t) = \mathcal{B}x(t)+K^{-1}Q\mathcal{C}x(t)=\Theta(0)x_-(0)+K^{-1}Q\Lambda(0)x_+(0)$.

\medskip

2.~Assume that $K\in\mathbb{C}^{n_-\times n_-}$ is not invertible. We show that $A_{\Delta}\colon D(A_{\Delta})\rightarrow \Ls^2_{|\Delta|}(0,\infty)$ is not a generator. Assume the contrary. Since $\rank K<n_-$ but $\rank[KQ]=n_{-}$ we see that $\rank Q\not=0$. Therefore also $\rank Q^{\star}\not=0$. This means that there is $v\in\mathbb{C}^{n_-}$ such that $Q^{\star}v\not=0$. We obtain that
$$
q^{\star}:=v^{\star}Q=(Q^{\star}v)^{\star}\not=0
$$
and we assume w.l.o.g.~that $q_1\not=0$. Now we select $g\in C_c^{\infty}[0,\infty)$ with $g(0)=0$ and $g(\xi)\not=0$ for $\xi\in(0,1)$ and put $x_0:=[g/\lambda_1\;0\cdots0]^T$ where $\lambda_1$ is the first entry of $\Lambda$. Then, $\mathcal{H}x_0=[g\;0\cdots0]^T\in \AC[0,\infty)$, $(\mathcal{H}x_0)'=[g'\;\;0\cdots0]^T\in \Ls^2_{|\Delta|}(0,\infty)$ and $K\Theta(0)x_-(0)+Q\Lambda(0)x_+(0)=0$ since $x(0)=0$ holds. This means that $x_0\in D(A_{\Delta})$ holds. By our assumption there exists a classical solution, i.e., a function $x=[x_+(\xi,t)\;\;x_-(\xi,t)]^T$ with
\begin{equation}\label{EQ-CLASS}
x(\cdot,0)=x_0,\;\;\frac{\partial}{\partial t}x(\xi,t)=\frac{\partial}{\partial t}\bigl(\mathcal{H}(\xi)x(\xi,t)\bigr)\;\text{ and }\;K\Theta(0)x_-(0,t)+Q\Lambda(0)x_+(0,t)=0
\end{equation}
for all $t\geqslant0$. Reading the above coordinate-wise and using the notation
$$
x_+=[x_{+,1}\;\cdots\;x_{+,n_+}]^T\;\text{ and }\; x_-=[x_{-,1}\;\cdots\;x_{-,n_-}]^T
$$
we can employ Proposition \ref{PROP-SEMI-PLUS} and Proposition \ref{PROP-SEMI-MINUS} to compute the coordinates of $x_+(\xi,t)$ and $x_-(\xi,t)$ explicitly to see that
$$
x(\xi,t)= \left[\hspace{-3pt}\begin{array}{r}x_{+,1}(\xi,t)\vspace{-3pt}\\\vdots\vspace{-5pt}\hspace{22pt}\\x_{+,n_+}(\xi,t)\\x_{-,1}(\xi,t)\vspace{-3pt}\\\vdots\vspace{-5pt}\hspace{22pt}\\x_{-,n_-}(\xi,t)\end{array}\hspace{-3pt}\right]
= \begin{bmatrix}\frac{\lambda_1(\xi+\mu_{\lambda_1}(\xi,t))}{\lambda_1(\xi)}x_{+,1}^{(0)}(\xi+\mu_{\lambda_1}(\xi,t))\\0\vspace{-4pt}\\\vdots\\0\end{bmatrix}
=\begin{bmatrix}\frac{g(\xi+\mu_{\lambda_1}(\xi,t))}{\lambda_1(\xi)}\\0\vspace{-4pt}\\\vdots\\0\end{bmatrix}
$$
holds. Here, we used that $x_{+,1}^{(0)}(\xi)=g(\xi)/\lambda_1(\xi)$ and $x_{+,k}^{(0)}(\xi)=x_{-,j}^{(0)}(\xi)=0$ hold for $k=2,\dots,n_+$ and $j=1,\dots,n_-$. Since we have $x_-(0,t)=0$, we get from the last equation in \eqref{EQ-CLASS} that
$$
0=Q\Lambda(0)x_+(0,t)=Q\begin{bmatrix}\lambda_1(0)\frac{g(0+\mu_{\lambda_1}(0,t))}{\lambda_1(0)}\\0\vspace{-4pt}\\\vdots\\0\end{bmatrix}=\begin{bmatrix}g(0+\mu_{\lambda_1}(0,t))\\0\vspace{-4pt}\\\vdots\\0\end{bmatrix}
$$
holds for all $t\geqslant0$. Multiplying from the left with $q^{\star}=[\bar{q}_1\;\cdots\;\bar{q}_{n_+}]$ yields
$$
0=[\bar{q}_1\;\cdots\;\bar{q}_{n_+}]\begin{bmatrix}g(\xi+\mu_{\lambda_1}(\xi,t))\\0\vspace{-4pt}\\\vdots\\0\end{bmatrix}=\bar{q}_1g(\mu_{\lambda_1}(0,t))=\bar{q}_1g(p_{\lambda_1}^{-1}(t))
$$
where the right hand side is non-zero if we select $t>0$ sufficiently small. Indeed, $\bar{q}_1\not=0$ holds by our assumptions, $\lim_{t\rightarrow0}p_{\lambda_1}^{-1}(t)=0$ and $g|_{(0,1)}\not=0$, implies the latter. Contradiction.
\end{proof}

\smallskip

\begin{rem}\label{REM-NULL} In all the $\mathbb{C}^n$-valued results, starting with Proposition \ref{PROP-DIAG}, we assumed that $n_+$ and $n_-$ are strictly positive. This had three reasons. Firstly, if $n_+=0$, or $n_-=0$, then we would need to put $\Delta=\Theta$, or $\Delta=\Lambda$ respectively, instead of $\Delta=\diag(\Lambda,\Theta)$, so all propositions would require corresponding definitions by cases. Secondly, when we consider the boundary control system $(\Sigma)$ in Lemma \ref{BCS-LEM} and Lemma \ref{LEM-BCS-WELLPOSED}, then allowing $n_+=0$ or $n_-=0$ would mean that we enter the pathological situation of the output or the input space being the zero space $\mathbb{C}^0=\{0\}$. Finally, if $n_+=0$ or $n_-=0$, then Proposition \ref{WEISS} needs to be modified as follows to remain true. Firstly, we notice that in both cases, $n_+=0$ or $n_-=0$, the matrix $Q\in\mathbb{C}^{n_-\times n_+}$ has to be removed from the statement. If $n_+=0$, then we need to put $[K\;\;Q]=K\in\mathbb{C}^{n_-\times n_-}$ and if $n_-=0$, then $[K\;\;Q]$ has to be omitted at all.

\medskip

\begin{compactitem}

\item[(i)] Let $n_+=0$ and $n=n_-\geqslant1$. Let $\Delta=\diag(\uptheta_1,\dots,\uptheta_{n_-})\colon[0,\infty)\rightarrow\mathbb{C}^{n\times n}$ be given such that $\uptheta_j\colon[0,\infty)\rightarrow(-\infty,0)$ is continuous and such that $\frac{1}{\uptheta_j}\not\in \Ls^1(0,\infty)$ for $j=1,\dots,n_-$. Let $K\in\mathbb{C}^{n_-\times n_-}$ have rank $n_-$. Then
\begin{eqnarray*}
\phantom{XXXX}A_{\Delta}x &= &(\Delta{}x)'\\
\phantom{XXXX}D(A_{\Delta}) &= &\bigl\{x\in \Ls^2_{|\Delta|}(0,\infty)\:;\:\Delta{}x\in \AC[0,\infty), \; (\Delta{}x)'\in \Ls^2_{|\Delta|}(0,\infty)\text{ and }K\Theta(0)x(0)=0\bigr\}
\end{eqnarray*}
generates a $\Cnull$-semigroup.

\vspace{5pt}

\item[(ii)] Let $n_-=0$ and $n=n_+\geqslant1$. Let $\Delta=\diag(\lambda_1,\dots,\lambda_{n_-})\colon[0,\infty)\rightarrow\mathbb{C}^{n\times n}$ be given such that $\lambda_k\colon[0,\infty)\rightarrow(-\infty,0)$ is continuous and such that $\frac{1}{\lambda_k}\not\in \Ls^1(0,\infty)$ for $k=1,\dots,n_+$. Then
\begin{eqnarray*}
\phantom{XXXX}A_{\Delta}x &= &(\Delta{}x)'\\
\phantom{XXXX}D(A_{\Delta}) &= &\bigl\{x\in \Ls^2_{|\Delta|}(0,\infty)\:;\:\Delta{}x\in \AC[0,\infty), \; (\Delta{}x)'\in \Ls^2_{|\Delta|}(0,\infty)\bigr\}
\end{eqnarray*}
generates a $\Cnull$-semigroup.
\end{compactitem}

\medskip

In (i), $D(A_{\Delta})$ remains unchanged if we put $K=I$. Thus, in both cases it is enough to proceed as in the proof of Proposition \ref{PROP-DIAG} but without the $\lambda_k$'s and the $\uptheta_j$'s, respectively. Notice that we do not need to assume that $\Delta$ is bounded as we do not use Lemma \ref{LEM-BCS-WELLPOSED} and Proposition \ref{WEISS} in this special case.
\end{rem}

\medskip

Finally we are able to treat general port-Hamiltonian partial differential equations. In Theorem \ref{PHS-THM} below we consider the case $P_0=0$ and in Corollary \ref{PHS-COR} we then allow $P_0\in\mathbb{C}^{n\times n}$ to be arbitrary. This will require an additional assumption on the Hamiltonian.

\bigskip

\begin{thm}\label{PHS-THM} Let $n\geqslant1$. Let $P_1\in\mathbb{C}^{n\times n}$ be Hermitian and invertible. Let $\mathcal{H}\colon[0,\infty)\rightarrow\mathbb{R}^{n\times n}$ be continuously differentiable, positive and Hermitian. Let $\Delta$, $S\colon[0,\infty)\rightarrow\mathbb{C}^{n\times{}n}$ be continuously differentiable such that $\Delta$ is diagonal and
$$
P_1\mathcal{H}(\xi)=S(\xi)^{-1}\Delta(\xi)S(\xi)
$$ 
holds for all $\xi\in[0,\infty)$. Let $n_-\geqslant1$ be the number of negative eigenvalues of $P_1$ and let $Z^-(0)$ be the span of the eigenvectors of $P_1\mathcal{H}(0)$ that correspond to negative eigenvalues. Let $n_+=n-n_-\geqslant0$ be the number of positive eigenvalues of $P_1$. Let $W_B\in\mathbb{C}^{n_-\times n}$ and let
$$
\begin{bmatrix}U_1&U_2\end{bmatrix}=W_B\mathcal{H}(0)S(0)^{-1}
$$
with $U_1\in\mathbb{C}^{n_-\times{}n_+}$, $U_2\in\mathbb{C}^{n_-\times{}n_-}$. If $n_+=0$, then let $U_2=W_B\mathcal{H}(0)S(0)^{-1}\in\mathbb{C}^{n_-\times{}n_-}$. Assume

\vspace{5pt}

\begin{compactitem}

\item[(a)] that $\Delta\colon[0,\infty)\rightarrow\mathbb{C}^{n\times n}$ is bounded,

\vspace{3.5pt}

\item[(b)] that $S\colon \Ls^2_{\mathcal{H}}(0,\infty)\rightarrow \Ls^2_{|\Delta|}(0,\infty)$, $x\mapsto Sx$ is an isomorphism of Banach spaces,

\vspace{2pt}

\item[(c)] that $B\colon \Ls^2_{|\Delta|}(0,\infty)\rightarrow \Ls^2_{|\Delta|}(0,\infty)$, $g\mapsto Bg:=S(S^{-1})'\Delta{}g$ is a bounded operator,

\vspace{3pt}

\item[(d)] that $C\colon \Ls^2_{\mathcal{H}}(0,\infty)\rightarrow \Ls^2_{|\Delta|}(0,\infty)$, $x\mapsto Cx:=S'P_1\mathcal{H}x$ is well-defined,

\vspace{4pt}

\item[(e)] that $\operatorname{rk}W_B=n_-$.

\end{compactitem}

\vspace{5pt}

Then for\vspace{-7pt}
\begin{eqnarray*}
A_{\mathcal{H}}x &= &P_1(\mathcal{H}x)'\\
D(A_{\mathcal{H}}) &= &\bigl\{x\in \Ls^2_{\mathcal{H}}(0,\infty)\:;\:\mathcal{H}x\in \AC[0,\infty), \; (\mathcal{H}x)'\in \Ls^2_{\mathcal{H}}(0,\infty)\text{ and }W_B(\mathcal{H}x)(0)=0\bigr\}.
\end{eqnarray*}

\vspace{-5pt}

the following are equivalent.

\begin{compactitem}

\vspace{5pt}

\item[(i)] The operator $A_{\mathcal{H}}\colon D(A_{\mathcal{H}})\rightarrow \Ls^2_{\mathcal{H}}(0,\infty)$ generates a $\Cnull$-semigroup.

\vspace{3pt}

\item[(ii)] The matrix $U_2$ is invertible.

\vspace{4pt}

\item[(iii)] We have $W_B\mathcal{H}(0)Z^{-1}(0)=\mathbb{C}^{n_-}$.
\end{compactitem}
\end{thm}
\begin{proof} Since
$$
\sigma(P_1\mathcal{H}(\xi))=\sigma(P_1\mathcal{H}(\xi)^{1/2}\mathcal{H}(\xi)^{1/2})=\sigma(\mathcal{H}(\xi)^{1/2}P_1\mathcal{H}(\xi)^{1/2})=\sigma(\mathcal{H}(\xi)^{1/2}P_1(\mathcal{H}(\xi)^{1/2})^{\star})
$$
holds, for every $\xi\in[0,\infty)$ the matrix $P_1\mathcal{H}(\xi)$ has the same number of positive resp.~negative eigenvalues as $P_1$ by Sylvester's law of inertia, see, e.g., \cite[Definition 4.5.4 and Theorem 4.5.7]{RH}. The number of negative eigenvalues of $P_1\mathcal{H}(\xi)$ is thus $n_-$ and the number of positive eigenvalues is $n_+$; both are independent of $\xi\in[0,\infty)$.

\smallskip

Assume first that $n_-,\:n_+\geqslant1$ holds. By rearranging the coordinates in $\mathbb{C}^n$ we may assume that $\Delta$ is of the form considered in Proposition \ref{PROP-DIAG}. By our assumptions the entries are strictly positive, resp.~strictly negative continuous functions and as they are all bounded by (a), their reciprocals are not in $\Ls^1(0,\infty)$. For $g\in \Ls^2_{|\Delta|}(0,\infty)$ we write $g(\xi)=[g_+(\xi)\;\;g_-(\xi)]^T$ with $g_+(\xi)\in\mathbb{C}^{n_+}$ and $g_-(\xi)\in\mathbb{C}^{n_-}$. Then we have
\begin{eqnarray*}
W_B(\mathcal{H}S^{-1}g)(0) & = & W_B\mathcal{H}(0)S^{-1}(0)g(0)\\
&=& \begin{bmatrix}U_1&U_2\end{bmatrix}\begin{bmatrix}g_+(0)\\g_-(0)\end{bmatrix}\\
&= & U_1g_+(0)+U_2g_-(0) + U_2g_-(0)\\
& = & U_1\Lambda(0)^{-1}(\Lambda(0)g_+(0))+U_2\Theta(0)^{-1}(\Theta(0)g_-(0)).\phantom{\int}\\
& = & Q(\Lambda(0)g_+(0))+K(\Theta(0)g_-(0)).
\end{eqnarray*}
with $Q:=U_1\Lambda(0)^{-1}$ and $K:=U_2\Theta(0)^{-1}$. By (e) we have 
$$
\operatorname{rank}\,[K\;\;Q]=\operatorname{rank}\,[U_1\;\;U_2]{\textstyle\begin{bmatrix}\Lambda(0)^{-1}&0\\0&\Theta(0)^{-1}\end{bmatrix}}=\operatorname{rank}\,[U_1\;\;U_2]=\operatorname{rank}\,W_B=n_-
$$
and therefore by Proposition \ref{WEISS}, the operator $A_{\Delta}\colon D(A_{\Delta})\rightarrow \Ls^2_{|\Delta|}(0,\infty)$ with
\begin{eqnarray*}
A_{\Delta}x &= &(\Delta{}x)'\\
D(A_{\Delta}) &= &\bigl\{g\in \Ls^2_{|\Delta|}(0,\infty)\:;\:\Delta{}g\in \AC[0,\infty), \; (\Delta{}g)'\in \Ls^2_{|\Delta|}(0,\infty)\text{ and }\;W_B(\mathcal{H}S^{-1}g)(0)=0\bigr\}
\end{eqnarray*}
generates a $\Cnull$-semigroup if and only if $K$ is invertible. Since $\Theta(0)^{-1}$ is invertible, the latter holds if and only if $U_2$ is invertible and this is true if and only if
\begin{equation}\label{EQIV}
\forall\:f\in\mathbb{C}^{n_-}\;\exists\:h\in\mathbb{C}^{n_-}\colon f=U_2h=[U_1\;U_2]{\textstyle\begin{bmatrix}0\\h\end{bmatrix}}=W_B\mathcal{H}(0)S^{-1}(0){\textstyle\begin{bmatrix}0\\h\end{bmatrix}}
\end{equation}
is valid. The last $n_-$ columns of $S^{-1}(0)$ are eigenvectors of $P_1\mathcal{H}(0)$ that correspond to negative eigenvalues of $P_1\mathcal{H}(0)$. Consequently, $S^{-1}(0)[0\;\;h]^T\in Z^{-}(0)$ and thus \eqref{EQIV} is equivalent to $W_B\mathcal{H}(0)Z^-(0)=\mathbb{C}^{n_-}$.

\smallskip

By (b) we have the isomorphism
$$
S\colon \Ls^2_{\mathcal{H}}(0,\infty)\longrightarrow \Ls^2_{|\Delta|}(0,\infty),\;\;x\mapsto Sx=:g
$$
which establishes a bijection between $\AC[0,\infty)\cap \Ls^2_{\mathcal{H}}(0,\infty)$ and $\AC[0,\infty)\cap \Ls^2_{|\Delta|}(0,\infty)$ as $S$ is continuously differentiable. For $g\in D(A_{\Delta})$ we therefore can compute locally
\begin{equation}\label{VAR-TRANS}
\begin{array}{rcl}\vspace{4pt}
SA_{\mathcal{H}}S^{-1}g & = & SP_1(\mathcal{H}S^{-1}g)'= S(P_1\mathcal{H}S^{-1}g)'= S(S^{-1}\Delta{}SS^{-1}g)'\\
& = & S(S^{-1})'\Delta{}g+SS^{-1}(\Delta{}g)'=  A_{\Delta}g + Bg
\end{array}
\end{equation}
where $B\colon \Ls^2_{|\Delta|}(0,\infty)\rightarrow \Ls^2_{|\Delta|}(0,\infty)$ is bounded by (c). This means that $A_{\Delta}\colon D(A_{\Delta})\rightarrow \Ls^2_{|\Delta|}(0,\infty)$ generates a $\Cnull$-semigroup if and only if $SA_{\mathcal{H}}S^{-1}\colon D(A_{|\Delta|})\rightarrow \Ls^2_{|\Delta|}(0,\infty)$ generates a $\Cnull$-semigroup. In order to conclude that this is equivalent to $A_{\mathcal{H}}\colon D(A_\mathcal{H})\rightarrow \Ls^2_{\mathcal{H}}(0,\infty)$ being a generator, it remains to see that
$$
S(D(A_{\mathcal{H}}))=D(A_{|\Delta|}),
$$
which then automatically implies $S^{-1}(D(A_{|\Delta|}))=D(A_{\mathcal{H}})$. Indeed, for $g\in D(A_{|\Delta|})$ our computation in \eqref{VAR-TRANS} shows $A_{\Delta}g+Bg\in \Ls^2_{|\Delta|}(0,\infty)$ and thus
$$
P_1(\mathcal{H}S^{-1}g)'=A_{\mathcal{H}}S^{-1}g=S^{-1}A_{\Delta}g+S^{-1}Bg\in \Ls^2_{\mathcal{H}}(0,\infty)
$$
holds, which means that $S^{-1}g\in D(A_{\mathcal{H}})$. For the converse let $x\in D(A_{\mathcal{H}})$ be given. We see that
$$
(\Delta{}Sx)'=(SP_1\mathcal{H}S^{-1}Sx)'=(SP_1\mathcal{H}x)'= S'P_1\mathcal{H}x+SP_1(\mathcal{H}x)'=Cx+SP_1(\mathcal{H}x)'\in \Ls^2_{|\Delta|}(0,\infty)
$$
is true by (d) and it follows that $Sx\in D(A_{\Delta})$.

\medskip

If $n_+=0$, then our assumptions imply already that $U_2$ is invertible and also that $W_B\mathcal{H}(0)Z^{-1}(0)=\mathbb{C}^{n_-}$ holds. It has thus to be shown that $A_{\mathcal{H}}$ always generates a $\Cnull$-semigroup. To see this we can repeat the above proof but delete $U_1$, $g_+$ and $\Lambda$ wherever they occur. Moreover, we apply Remark \ref{REM-NULL}(i) instead of Proposition \ref{WEISS} to get the desired result.
\end{proof}

\smallskip

\begin{rem}\label{REM-NULL-2}\begin{compactitem}\item[(i)]Theorem \ref{PHS-THM} remains true if the number $n_-$ of negative eigenvalues of $P_1$ is zero. In this case we have to delete the boundary condition $W_B(\mathcal{H}x)(0)=0$ from the domain of $D(A_{\mathcal{H}})$ and understand conditions (ii) and (iii) to be always true. That then (i) is always valid follows by repeating the second part of the proof of Theorem \ref{PHS-THM} and using Remark \ref{REM-NULL}(ii) instead of Proposition \ref{WEISS}.

\vspace{3pt} 

\item[(ii)] In the case that $P_1\mathcal{H}$ is diagonal and bounded, Proposition \ref{WEISS} contains already a generation result for the operator considered Theorem \ref{PHS-THM}. Indeed, in Theorem \ref{PHS-THM} we could then select $\Delta=P_1\mathcal{H}$ and $S=I$ which would show that the conditions in Theorem \ref{PHS-THM}(a)-(d) hold. The equivalence of Theorem \ref{PHS-THM}(i)-(ii) recovers the statement of Proposition \ref{WEISS} with
$$
[U_1\;\;U_2]=W_B\mathcal{H}(0)=W_B\diag(\Delta(0),\Theta(0))=[Q\;\;K]\diag(\Delta(0),\Theta(0))=Q\Delta(0)+K\Theta(0)
$$
and thus $U_2=K\Theta(0)$ which is invertible if and only if $K$ is invertible. The difference in this case is that we assumed in Theorem \ref{PHS-THM} that $\mathcal{H}$ is continuously differentiable whereas in Proposition \ref{WEISS} it is enough if $\Delta=P_1\mathcal{H}$ is continuous.

\vspace{3pt} 

\item[(iii)] The assumption that $P_1\mathcal{H}$ can be diagonalized in a smooth way is not very restrictive, see \cite[Remark 1.6.1]{JMZ2015} and Kato \cite[Section II]{Kato}. For the wave equation with continuously differentiable coefficients the matrices $S$ and $S^{-1}$ can easily be seen to be continuously differentiable, see Section \ref{SEC:EX}. The other conditions in Theorem \ref{PHS-THM} are more restrictive, but indeed they allow for instance to treat the wave equation with a Hamiltonian that is neither bounded nor bounded away from zero, see Example \ref{EX:3}.
\end{compactitem}
\end{rem}

\bigskip

In the situation of Theorem \ref{PHS-THM} we can perturb the generator $A_{\mathcal{H}}\colon D(A_{\mathcal{H}})\rightarrow \Ls^2_{\mathcal{H}}(0,\infty)$ with any bounded operator $D\colon \Ls^2_{\mathcal{H}}(0,\infty)\rightarrow \Ls^2_{\mathcal{H}}(0,\infty)$ and $A_{\mathcal{H}}+D\colon D(A_{\mathcal{H}})\rightarrow \Ls^2_{\mathcal{H}}(0,\infty)$ will be again a generator. The special case where $D$ is given by the multiplication with $P_0\mathcal{H}$ where $P_0\in\mathbb{C}^{n\times n}$ leads us to the situation of a port-Hamiltonian partial differential equation as considered in Section \ref{SEC:INTRO}.

\smallskip

\begin{cor}\label{PHS-COR} Let $n\geqslant1$. Let $P_1\in\mathbb{C}^{n\times n}$ be Hermitian and invertible. Let
$P_0\in\mathbb{C}^{n\times n}$ be arbitrary. Let $\mathcal{H}\colon[0,\infty)\rightarrow\mathbb{R}^{n\times n}$ be continuously differentiable, positive and Hermitian. Let $n_-$, $\Delta$, $S$, $W_B$, $Z^{-}(0)$ and $U_2$ be as in Theorem \ref{PHS-THM} and assume additionally, that $\mathcal{H}\colon \Ls^2_{\mathcal{H}}(0,\infty)\rightarrow \Ls^2_{\mathcal{H}}(0,\infty)$ is bounded. Then\vspace{-4pt}
\begin{eqnarray*}
A_{\mathcal{H}}x &= &P_1(\mathcal{H}x)'+P_0\mathcal{H}x\\
D(A_{\mathcal{H}}) &= &\bigl\{x\in \Ls^2_{\mathcal{H}}(0,\infty)\:;\:\mathcal{H}x\in \AC[0,\infty), \; (\mathcal{H}x)'\in \Ls^2_{\mathcal{H}}(0,\infty)\text{ and }W_B(\mathcal{H}x)(0)=0\bigr\}.
\end{eqnarray*}

\vspace{-5pt}

generates a $\Cnull$-semigroup if and only if the equivalent conditions in Theorem \ref{PHS-THM}(ii)--(iii)  are satisfied.\hfill\qed
\end{cor}

\smallskip

Notice that the additional assumption of Corollary \ref{PHS-COR} is automatically satisfied if $\mathcal{H}\colon[0,\infty)\rightarrow\mathbb{C}^{n\times n}$ is bounded. In the case that $P_1\mathcal{H}$ is diagonal and bounded, and Proposition \ref{WEISS} is applied instead of Theorem \ref{PHS-THM}, we get immediately a corresponding result for the port-Hamiltonian system with arbitrary $P_0\in\mathbb{C}^{n\times n}$.

\bigskip


\section{Well-posedness}\label{SEC:WP}\smallskip

In Proposition \ref{WEISS} we used the Weiss Theorem, see Theorem \ref{WEISS-THM} in Section \ref{INTERLUDE}, to characterize the linear boundary conditions that turn the operator $A_{\Delta}\colon \Ls^2_{\Delta}(0,\infty)\rightarrow \Ls^2_{\Delta}(0,\infty)$ into the generator of a $\Cnull$-semigroup. Our proof showed however more, namely that the boundary control system $(\Sigma')$, see p.~\pageref{SIGMA-PRIME}, is well-posed. Our aim in this section is to extend the latter to the general form of a port-Hamiltonian partial differential equation and to general boundary control and measurement. We follow the ideas of \cite[Chapter 13.2 and 13.4]{JZ} and consider the port-Hamiltonian system
\begin{subequations}\label{eq:PHS}
\begin{align}
\displaystyle\frac{\partial x}{\partial t}(\xi,t)&=P_1\frac{\partial}{\partial\xi}(\mathcal{H}(\xi)x(\xi,t))+P_0(\mathcal{H}(\xi)x(\xi,t))\label{PHS-a}\\
x(\xi,0)&=x_0(\xi)\label{PHS-b}\\[4.5pt]
u(t)&=W_{B,1}\mathcal{H}(0)x(0,t)\label{PHS-c}\\[4.5pt]
0&=W_{B,2}\mathcal{H}(0)x(0,t)\label{PHS-d}\\[4.5pt]
y(t)&=W_C\mathcal{H}(0)x(0,t)\label{PHS-e}
\end{align}
\end{subequations}
for $\xi\in[0,\infty)$ and $t\geqslant0$, where $W_{B,1}\in\mathbb{C}^{p\times n}$, $W_{B,2}\in\mathbb{C}^{(n_--p)\times n}$ with $1\leqslant p\leqslant n_-$ and $W_C\in\mathbb{C}^{q\times n}$ for $1\leqslant q\leqslant n_+$ with $n_+=n-n_-$. If $p=n_-$ we understand that equation \eqref{PHS-d} is dropped. Above, $n_-$ stands, as in Section \ref{SEC:GEN}, for the number of negative eigenvalues of $P_1$ which we throughout this section assume to be strictly positive. Notice that Section \ref{SEC:GEN} showed that without this assumptions we get a generator without boundary conditions, and therefore a discussion of boundary control and measurement leads only to pathological cases. Similar reasons lead to the assumption $p$, $q\geqslant1$ that we made above.

\smallskip

\begin{lem}\label{PHS-BS} Let the general assumptions of Theorem \ref{PHS-THM} be satisfied with $W_B=[W_{B,1}\;\;W_{B,2}]^T\in\mathbb{C}^{n_-\times n}$ and let $\mathcal{H}\colon \Ls^2_{\mathcal{H}}(0,\infty)\rightarrow \Ls^2_{\mathcal{H}}(0,\infty)$ be bounded. Assume that the equivalent conditions in Theorem \ref{PHS-THM}(i)--(iii) are satisfied. We define $\mathcal{A}\colon D(\mathcal{A})\rightarrow \Ls^2_{\mathcal{H}}(0,\infty)$, $\mathcal{B}\colon D(\mathcal{B})\rightarrow\mathbb{C}^{p}$ and $\mathcal{C}\colon D(\mathcal{C})\rightarrow\mathbb{C}^{q}$ via
$$
\mathcal{A}x=P_1(\mathcal{H}x)'+P_0\mathcal{H}x,\;\;\mathcal{B}x=W_{B,1}(\mathcal{H}x)(0),\;\;\mathcal{C}x=W_{C}(\mathcal{H}x)(0)
$$
where the domains are given by $D(\mathcal{C})=D(\mathcal{B})=D(\mathcal{A})=\{x\in \Ls^2_{\mathcal{H}}(0,\infty)\:;\:\mathcal{H}x\in \AC[0,\infty),\;(\mathcal{H}x)'\in \Ls^2_{\mathcal{H}}(0,\infty),\;W_{B,2}(\mathcal{H}x)(0)=0\}$. Then
$$
(\Sigma_{\mathcal{H}})\;\;\;
\begin{cases}
\;\dot{x}(t)=\mathcal{A}x,\\
\;u(t)=\mathcal{B}x,\\
\;y(t)=\mathcal{C}x,
\end{cases}
$$
defines a boundary control system in the sense of Definition \ref{JZ-ASS}.
\end{lem}

\begin{proof} By Corollary \ref{PHS-COR}, $\mathcal{A}|_{\operatorname{ker}\mathcal{B}}\colon D(\mathcal{A})\cap\operatorname{ker}\mathcal{B}\rightarrow \Ls^2_{\mathcal{H}}(0,\infty)$ generates a $\Cnull$-semigroup. Now we select $\varphi\in C^{\infty}[0,\infty)$ with compact support and such that $\varphi(0)=1$. Since $\operatorname{rank}W_B=n_-$, we can put $S:=W_B^{\star}(W_BW_B^{\star})^{-1}\diag(I_p,0)\in\mathbb{C}^{n\times n_-}$, where $I_p$ is the identity matrix in $\mathbb{C}^{p\times p}$ and $0$ denotes the zero matrix in $\mathbb{C}^{(n_--p)\times (n_--p)}$. We get
$$
W_BS=\begin{bmatrix}W_{B,1}\\W_{B,2}\end{bmatrix}\begin{bmatrix}S_{1}&S_{2}\end{bmatrix}=\begin{bmatrix}I_p&0\\0&0\end{bmatrix}\in\mathbb{C}^{n_-\times n_-}
$$
where $S_{1}\in\mathbb{C}^{n\times p}$. If $p=n_-$ we omit the zeros and $S_2$ in the equation above. We define
$$
B\colon\mathbb{C}^p\rightarrow \Ls^2_{\mathcal{H}}(0,\infty),\;\;(Bu)(\xi)=\mathcal{H}(0)^{-1}S_{1}\varphi(\xi)u
$$
which is a linear operator with values in $D(\mathcal{A})$ due to the properties of $\varphi$ and since
$$
W_{B,2}(\mathcal{H}Bu)(0)=W_{B,2}\mathcal{H}(0)\mathcal{H}(0)^{-1}S_{1}\varphi(0)u=W_{B,2}S_{1}u=0
$$
holds. Finally we have
$$
\mathcal{B}Bu=W_{B,1}\mathcal{H}(0)(Bu)(0)=W_{B,1}\mathcal{H}(0)\mathcal{H}(0)^{-1}S_{11}\varphi(0)u=W_{B,1}S_{1}u=u
$$
for every $u\in\mathbb{C}^p$. Moreover, $\mathcal{A}B\colon\mathbb{C}^{p}\rightarrow \Ls^2_{|\Delta|}(0,\infty)$ and $\mathcal{C}\colon D(\mathcal{C})\rightarrow\mathcal{C}^{q}$ are linear and continuous. Consequently $(\Sigma_{\mathcal{H}})$ defines a boundary control system in the sense of Definition \ref{JZ-ASS}.
\end{proof}

\smallskip

\begin{thm}\label{WP-THM} Let the general assumptions of Theorem \ref{PHS-THM} be satisfied with $W_B=[W_{B,1}\;\;W_{B,2}]^T\in\mathbb{C}^{n_-\times n}$ and let $\mathcal{H}\colon \Ls^2_{\mathcal{H}}(0,\infty)\rightarrow \Ls^2_{\mathcal{H}}(0,\infty)$ be bounded. Assume that the equivalent conditions in Theorem \ref{PHS-THM}(i)--(iii) are satisfied. Then the port-Hamiltonian system \eqref{eq:PHS}, more formally the boundary control system of Lemma \ref{PHS-BS}, is well-posed.
\end{thm}
\begin{proof} Due to our assumptions, the map $P_0\mathcal{H}\colon \Ls^2_{\mathcal{H}}(0,\infty)\rightarrow \Ls^2_{\mathcal{H}}(0,\infty)$ is bounded. In view of \cite[Lemma 13.1.14]{JZ} it is thus enough to prove the well-posedness of $(\Sigma_{\mathcal{H}})$ from Lemma \ref{PHS-BS} for $P_0=0$.

\smallskip

As in the proof of Theorem \ref{PHS-THM} we use the transformation of variables $g=Sx$ and use the notation established in Section \ref{SEC:GEN}. We define $\tilde{W}_{B,1}:=W_{B,1}P_1^{-1}S^{-1}(0)$, $\tilde{W}_{B,2}:=W_{B,2}P_1^{-1}S^{-1}(0)$ and $\tilde{W}_{C}:=W_{C}P_1^{-1}S^{-1}(0)$. Then the boundary conditions \eqref{PHS-c}--\eqref{PHS-e} can be transformed into $u(t)=\tilde{W}_{B,1}(\Delta{}g)(0)$, $0=\tilde{W}_{B,2}(\Delta{}g)(0)$ and $y(t)=\tilde{W}_{C}(\Delta{}g)(0)$. Since $P_1^{-1}S^{-1}(0)$ are invertible, $\tilde{W}_B=[\tilde{W}_{B,1}\;\;\tilde{W}_{B,2}]^T$ has rank $n_-$.

\medskip

The transformation of variables thus leads to the system
$$
(\Sigma_{\Delta})\;\;\;
\begin{cases}
\;\dot{g}(t)=\mathcal{A}g,\\
\;u(t)=\mathcal{B}g,\\
\;y(t)=\mathcal{C}g,
\end{cases}
$$
with $\mathcal{A}\colon D(\mathcal{A})\rightarrow \Ls^2_{|\Delta|}(0,\infty)$, $\mathcal{B}\colon D(\mathcal{B})\rightarrow\mathbb{C}^{p}$ and $\mathcal{C}\colon D(\mathcal{C})\rightarrow\mathbb{C}^{q}$ which are defined via
$$
\mathcal{A}g=(\Delta{}g)',\;\;\mathcal{B}g=\tilde{W}_{B,1}(\Delta{}g)(0),\;\;\mathcal{C}g=\tilde{W}_{C}(\Delta{}g)(0)
$$
where the domains are given by $D(\mathcal{C})=D(\mathcal{B})=D(\mathcal{A})=\{g\in \Ls^2_{|\Delta|}(0,\infty)\:;\:\Delta{}g\in \AC[0,\infty),\;(\Delta{}g)'\in \Ls^2_{\Delta{}}(0,\infty),\;\tilde{W}_{B,2}(\Delta{}g)(0)=0\}$. It is thus enough to show that $(\Sigma_{\Delta})$ is well-posed. 

\medskip

We define a third system
$$
(\tilde{\Sigma}_{\Delta})\;\;\;
\begin{cases}
\;\dot{\tilde{g}}(t)=\tilde{\mathcal{A}}\tilde{g},\\
\;\tilde{u}(t)=\tilde{\mathcal{B}}\tilde{g},\\
\;\tilde{y}(t)=\tilde{\mathcal{C}}\tilde{g},
\end{cases}
$$
via $\tilde{\mathcal{A}}\colon D(\tilde{\mathcal{A}})\rightarrow \Ls^2_{|\Delta|}(0,\infty)$, $\tilde{\mathcal{B}}\colon D(\tilde{\mathcal{B}})\rightarrow\mathbb{C}^p$ and $\tilde{\mathcal{C}}\colon D(\tilde{\mathcal{C}})\rightarrow\mathbb{C}^q$ given by
$$
\tilde{\mathcal{A}}\tilde{g}=(\Delta{}\tilde{g})',\;\;\tilde{\mathcal{B}}\tilde{g}=[\tilde{W}_{B,1}\;\;\tilde{W}_{B,2}]^T(\Delta{}\tilde{g})(0),\;\;\tilde{\mathcal{C}}\tilde{g}=\tilde{W}_{C}(\Delta{}\tilde{g})(0)
$$
where the domains are $D(\tilde{\mathcal{B}})=D(\tilde{\mathcal{A}})=\{\tilde{g}\in \Ls^2_{|\Delta|}(0,\infty)\:;\:\Delta{}\tilde{g}\in \AC[0,\infty),\;(\Delta{}\tilde{g})'\in \Ls^2_{\Delta{}}(0,\infty)\}$ and $D(\tilde{\mathcal{C}})=D(\mathcal{C})$. If $p=n_-$, then we put $(\tilde{\Sigma}_{\Delta})=(\Sigma_{\Delta})$. It can be seen by the same arguments as in Lemma \ref{PHS-BS} that the above is a boundary control system.

\smallskip

Assume now that $(\tilde{\Sigma}_{\Delta})$ is well-posed, cf.~efinition \ref{JZ-WP}. Then there exists $\tau>0$ and $m_\tau>0$ such that the estimate
\begin{equation}\label{EQ-WP}
\|\tilde{g}(\tau)\|_{\Ls^2_{|\Delta|}(0,\infty}^2+\int_0^{\tau}\|\tilde{y}(t)\|^2\dd t\leqslant m_{\tau}\Big(\|\tilde{g}_0\|_{\Ls^2_{|\Delta|}(0,\infty)}^2+\int_0^{\tau}\|\tilde{u}(t)\|^2\dd t\Big)
\end{equation}
is true for any $g_0\in D(\tilde{\mathcal{A}})$, $\tilde{u}\in C^2([0,\tau],\mathbb{C}^{n_-})$ with $\tilde{u}(0)=\tilde{\mathcal{B}}g_0$ where $\tilde{g}$ denotes the classical solution and $\tilde{y}=\tilde{\mathcal{C}}g$ is the corresponding output. We claim that \eqref{EQ-WP} holds for the system $(\Sigma_{\Delta})$. We select $\tau$, $m_{\tau}>0$ as above. Let $g_0\in D(\mathcal{A})\subseteq D(\tilde{\mathcal{A}})$, $u\in C^2([0,\tau],\mathbb{C}^p)$ and $y\in D(\mathcal{C})$ be given. We put $\tilde{g}_0=g_0$, $\tilde{u}=[u\;\;0]^T\in C^2([0,\tau],\mathbb{C}^{n_-})$, and $\tilde{y}=y$. Then the classical solution $g$ corresponding to $(\Sigma_{\Delta})$ coincides with the classical solution $\tilde{g}$ corresponding to $(\tilde{\Sigma}_{\Delta})$. Thus, $(\tilde{\Sigma}_{\Delta})$ is also well-posed.

\smallskip

It remains to show that $(\tilde{\Sigma}_{\Delta})$ is well-posed. We choose an invertible matrix $P\in\mathbb{C}^{n_-\times n_-}$ such that
$$
P\begin{bmatrix}\tilde{W}_{B,1}\\\tilde{W}_{B,2}\end{bmatrix}=[Q\;\;K]\in\mathbb{C}^{n_-\times n}
$$
with $Q\in\mathbb{C}^{n_-\times n_+}$ arbitrary and $K\in\mathbb{C}^{n_-\times n_-}$ invertible. The proof of Proposition \ref{WEISS} showed that the system 
$$
(\Sigma')\;\;\;
\begin{cases}
\;\dot{\tilde{g}}(t)=(\Delta{}\tilde{g})',\\
\;\tilde{u}(t)=[K^{-1}Q\;\;I_{n_-}](\Delta{}\tilde{g})(0),\\
\;\tilde{y}(t)=(\Delta{}\tilde{g})(0),,
\end{cases}
$$
is well-posed. As $P^{-1}K\colon\mathbb{C}^{n_-}\rightarrow\mathbb{C}^{n_-}$, $u\mapsto P^{-1}Ku$, is an isomorphism, and $\diag(W_C,0)\colon\mathbb{C}^{q}\rightarrow\mathbb{C}^{n_-}$, $y\mapsto\diag(W_C,0)y$ is continuous, we get that $(\tilde{\Sigma}_{\Delta})$ is well-posed.
\end{proof}


\section{Examples}\label{SEC:EX}\smallskip

\begin{ex}\label{EX:1} \textbf{(Weighted transport equation)} We consider the one-dimensional \textquotedblleft{}weighted transport equation\textquotedblright{} on $(0,\infty)$, i.e.,
$$
\frac{\partial{}x}{\partial{}t}(\xi,t)=-\frac{\partial\mathcal{H}x}{\partial\xi}(\xi,t),\;\;x(\xi,0)=x_0(\xi),
$$
where $\mathcal{H}$ is continuous, strictly positive and satisfies $\frac{1}{\mathcal{H}}\not\in \Ls^1(0,\infty)$. With $\uptheta=-\mathcal{H}$ we see that the above fits into the situation of Proposition \ref{PROP-SEMI-MINUS} which yields that
\begin{eqnarray*}
Ax &=&-\frac{\partial}{\partial\xi}(\mathcal{H}x) \\
D(A) &= &\bigl\{x\in \Ls^2_{\mathcal{H}}(0,\infty)\:;\:\mathcal{H}x\in \AC[0,\infty),\; (\mathcal{H}x)'\in \Ls^2_{\mathcal{H}}(0,\infty)\text{ and }(\mathcal{H}x)(0)=0\bigr\}
\end{eqnarray*}
generates a $\Cnull$-semigroup on $\Ls^2_{\mathcal{H}}(0,\infty)$. Similarly, the one-dimensional weighted transport equation on $(0,\infty)$ with positive sign gives rise to a $\Cnull$-semigroup via Proposition \ref{PROP-SEMI-PLUS}. Notice that in the latter case there is no boundary condition at zero.
\end{ex}

\medskip

\begin{ex} \textbf{(Transport equation on a network)} We consider a network of five transport equations, each defined on $[0,\infty)$, i.e.,
$$
\frac{\partial x_j}{\partial t}(\xi,t)=p_j\frac{\partial x_j}{\partial \xi}(\xi,t),\;\;x_j(\xi,0)=x_j^{(0)}(\xi)\,\text{ for }\xi\in[0,\infty),\;\;t\geqslant0 \text{ and }\;j=1,\dots,5,
$$
with $p_j=+1$ for $j=1,\,2,\,3$ and $p_j=-1$ for $j=4,\,5$ and coupled by the equations
$$
x_4(0,t)=x_2(0,t)-x_3(0,t)\;\text{ and }\;x_5(0,t)=x_2(0,t)-x_1(0,t)
$$
relating the inputs on the outgoing edges $4$ and $5$ with the output of the incoming edges $1$, $2$ and $3$.

\smallskip

\begin{center}
\begin{tikzpicture}
[nodeDecorate/.style={circle,draw,fill=black!100, inner sep=0pt,minimum width=6pt}]

\node (v) at (5,1) [nodeDecorate] {\scriptsize$v$};
\node (a) at (2,2) [] {\footnotesize$1$};
\coordinate (a1) at (4,2) [] {};
\node (b) at (2,1) [] {\footnotesize$2$};
\coordinate (b1) at (4,1) [] {};
\node (c) at (2,0) [] {\footnotesize$3$};
\coordinate (c1) at (4,0) [] {};
\node (d) at (8,0.3) [] {\footnotesize$5$};
\coordinate (d1) at (6,0.3) [] {};
\node (e) at (8,1.7) [] {\footnotesize$4$};
\coordinate (e1) at (6,1.7) [] {};

\tikzstyle{EdgeStyle}=[-,>=stealth]
\tikzstyle{LabelStyle}=[fill=white]

\Edge[label={\ding{234}}](a)(a1)\Edge[](a1)(v)
\Edge[label={\ding{234}}](b)(b1)\Edge[](b1)(v)
\Edge[label={\ding{234}}](c)(c1)\Edge[](c1)(v)
\Edge[label={\ding{234}}](d)(d1)\Edge[](d1)(v)
\Edge[label={\ding{234}}](e)(e1)\Edge[](e1)(v)
\end{tikzpicture}

\smallskip

{

\small

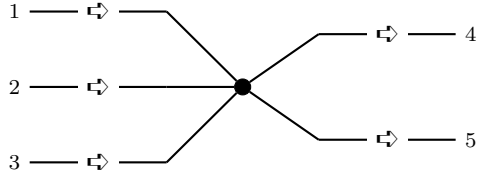
\textbf{Figure 1.\,}Network of five transport equations on $[0,\infty)$ which\\are coupled via boundary conditions imposed at a central node.

}
\end{center}

\medskip

We use the notation $x=[x_1\cdots x_5]^T$, put $P_1:=\diag(p_1,\dots,p_5)$, and consider the operator
\begin{eqnarray*}
Ax &= &P_1\frac{\partial}{\partial \xi}x \\
D(A) &= &\bigl\{x\in \Ls^2(0,\infty)\:;\:x\in \AC[0,\infty), \; x'\in \Ls^2(0,\infty)\text{ and }W_Bx(0)=0\bigr\}
\end{eqnarray*}
with the matrix
$$
W_B=\begin{bmatrix}
1&-1 &0 &0 &1 \\
0&-1 &1 &1 &0 \\
\end{bmatrix}.
$$
As $P_1\mathcal{H}=P_1$ is diagonal, $n_+=3$, and $n_-=2$, we are in the situation of Proposition \ref{WEISS} if we rewrite the boundary condition as
$$
0=W_Bx(0)=Qx_+(0)+Kx_-(0)=
\begin{bmatrix}
1 &-1 &0 \\
0 &-1 &1 \\
\end{bmatrix}
\begin{bmatrix}
x_1(0,t)\\
x_2(0,t)\\
x_3(0,t)
\end{bmatrix}
+
\begin{bmatrix}
0&1 \\
1&0 \\
\end{bmatrix}
\begin{bmatrix}
x_4(0,t)\\
x_5(0,t)\\
\end{bmatrix}
$$
with an invertible matrix $K$. It follows that $A\colon D(A)\rightarrow \Ls^2(0,\infty)$ generates a $\Cnull$-semigroup.

\smallskip

We remark that we can modify the above and consider a network of \textquotedblleft{}weighted transport equations\textquotedblright{} in the spirit of Example \ref{EX:1} by adding a Hamiltonian $\mathcal{H}=\diag(h_1,\dots,h_5)$ with functions $h_j\colon[0,\infty)\rightarrow\mathbb{R}$ that are continuous, strictly positive and bounded.

\smallskip

Notice that in the initial setting of unweighted transport equations we could also use Theorem \ref{PHS-THM} instead of Proposition \ref{WEISS} and obtain exactly the same result. In this case we see that $U_2=K$ is invertible or $W_BZ^{-1}(0)=W_B\mathbb{C}^{2}=\mathbb{C}^2$. In the weighted case, Theorem \ref{PHS-THM} would require continuously differentiable $h_j$'s, whereas Proposition \ref{WEISS} requires only continuity. Both results require boundedness.

\smallskip

The above example is related to the study of so-called metric graphs, see, e.g., Mugnolo \cite{Mug}. We mention that a result by Schubert et al.~\cite[Section 4.2]{BS}, with some slight adjustments, allows to treat the situation of even countable many coupled transport equations on $[0,\infty)$. Their theorem however characterizes the existence of unitary $\Cnull$-semigroups for which an equal number of outgoing and ingoing edges is necessary.
\end{ex}

\medskip

\begin{ex}\label{EX:3} \textbf{(Vibrating string)} Consider an undamped vibrating string of infinite length, i.e.,
\begin{equation}\label{VIB-1}
\frac{\partial^2w}{\partial{}^2t}(\xi,t)=-\frac{1}{\rho(\xi)}\frac{\partial}{\partial\xi}\Bigl(T(\xi)\frac{\partial{}w}{\partial\xi}(\xi,t)\Bigr)
\end{equation}
where $\xi\in[0,\infty)$ is the spatial variable, $w(\xi,t)$ is the vertical displacement of the string at place $\xi$ and time $t$, $T(\xi)>0$ is Young's modulus of the string, and $\rho(\xi)>0$ is the mass density. Both may vary along the string in a continuously differentiable way. We choose the momentum $x_1=\rho\frac{\partial w}{\partial t}$ and the strain $x_1=\frac{\partial w}{\partial \xi}$ as the state variables. Then, \eqref{VIB-1} can be written as
\begin{equation}\label{VIB-2}
\frac{\partial}{\partial t}\begin{bmatrix}x_1(\xi,t)\\x_2(\xi,t)\end{bmatrix}=\begin{bmatrix}0 &1\\1&0\end{bmatrix}\frac{\partial}{\partial\xi}\Bigl(\begin{bmatrix}\rho(\xi)^{-1} & 0\\0&T(\xi)\end{bmatrix}\begin{bmatrix}x_1(\xi,t)\\x_2(\xi,t)\end{bmatrix}\Bigr)
\end{equation}
which is of the form considered in Theorem \ref{PHS-THM} if we put
$$
P_1:=\begin{bmatrix}0 & 1\\1&0\end{bmatrix},\;\;P_0:=\begin{bmatrix}0 & 0\\0&0\end{bmatrix},\,\text{ and }\;\mathcal{H}(\xi):=\begin{bmatrix}\rho(\xi)^{-1} & 0\\0&T(\xi)\end{bmatrix}.
$$
Diagonalizing $P_1\mathcal{H}=S^{-1}\Delta S$ leads to
$$
S=\begin{bmatrix}(2\gamma)^{-1} & \rho/2\\-(2\gamma)^{-1}&\rho/2\end{bmatrix},\;\;S^{-1}=\begin{bmatrix}\gamma & -\gamma\\\rho^{-1}&\rho^{-1}\end{bmatrix},\,\text{ and }\;\Delta=\begin{bmatrix}\gamma & 0\\0&-\gamma\end{bmatrix},
$$
where $\gamma:=(T/\rho)^{1/2}$. In particular, we have $n_+=n_-=1$ and
$$
Z^-(0)=\operatorname{span}\begin{bmatrix}-\gamma(0)\\\rho(0)^{-1}\end{bmatrix}.
$$
We endow \eqref{VIB-2} with the boundary condition
$$
W_B(\mathcal{H}x)(0)=0
$$
given by the matrix $W_B=[w_1\;\;w_2]\in\mathbb{C}^{1\times2}$ of rank $1$, or, equivalently, the wave equation \eqref{VIB-1} with the boundary conditions
$$
\begin{bmatrix}w_1\;\;w_2\end{bmatrix}\begin{bmatrix}\rho\frac{\partial w}{\partial t}(0,t)\\\frac{\partial w}{\partial\xi}(0,t)\end{bmatrix}=0.
$$
The condition in Theorem \ref{PHS-THM}(iii), i.e., $W_B\mathcal{H}(0)Z^{-1}(0)=\mathbb{C}^{1}$, holds if and only if
$$
0\not=\begin{bmatrix}w_1 & w_2\end{bmatrix}
\begin{bmatrix}\rho(0)^{-1} & 0\\0&T(0)\end{bmatrix}
\begin{bmatrix}-\gamma(0)\\\rho(0)^{-1}\end{bmatrix}=
-w_1\rho(0)^{-1}\gamma(0)+w_2T(0)\rho(0)^{-1}
$$
which is equivalent to
\begin{equation}\label{COND-1}
w_1\gamma(0)\not=w_2T(0).
\end{equation}
The operator $B\colon \Ls^2_{|\Delta|}(0,\infty)\rightarrow \Ls^2_{|\Delta|}(0,\infty)$ from Theorem \ref{PHS-THM} is given by multiplication with
$$
S(S^{-1})'\Delta=
\begin{bmatrix}(2\gamma)^{-1} & \rho/2\\-(2\gamma)^{-1}&\rho/2\end{bmatrix}
\begin{bmatrix}\gamma' & -\gamma'\\(\rho^{-1})'&(\rho^{-1})'\end{bmatrix}
\begin{bmatrix}\gamma & 0\\0&-\gamma\end{bmatrix}
=\frac{\delta}{4\gamma\rho^2}\begin{bmatrix}1&1\\-1&-1\end{bmatrix} + \frac{\sigma\gamma}{2\rho}\begin{bmatrix}-1&1\\-1&1\end{bmatrix}
$$
where $\delta:=T'\rho-T\rho'$ and $\sigma:=\rho'$. The map $C\colon \Ls^2_{\mathcal{H}}(0,\infty)\rightarrow \Ls^2_{|\Delta|}(0,\infty)$ from Theorem \ref{PHS-THM} is given by multiplication with
$$
S'P_1\mathcal{H}=
\begin{bmatrix}((2\gamma)^{-1})' & \rho'/2\\-((2\gamma)^{-1})'&\rho'/2\end{bmatrix}
\begin{bmatrix}0 & 1\\1&0\end{bmatrix}
\begin{bmatrix}\rho(\xi)^{-1} & 0\\0&T(\xi)\end{bmatrix}=
\frac{\sigma}{2\rho}\begin{bmatrix}1 & 0\\1&0\end{bmatrix}+\frac{\delta T}{4\gamma^3\rho^2}\begin{bmatrix}0 & -1\\0&1\end{bmatrix}.
$$
The question of interest is, when the operator
\begin{eqnarray*}
A_{\mathcal{H}}x &= &P_1(\mathcal{H}x)'\\
D(A_{\mathcal{H}}) &= &\bigl\{x\in \Ls^2_{\mathcal{H}}(0,\infty)\:;\:\mathcal{H}x\in \AC[0,\infty), \; (\mathcal{H}x)'\in \Ls^2_{\mathcal{H}}(0,\infty)\text{ and }W_B(\mathcal{H}x)(0)=0\bigr\}.
\end{eqnarray*}
generates a $\Cnull$-semigroup. We discuss the following three scenarios.

\medskip

1.~For constant modulus $T(\xi)\equiv{}T>0$ and constant mass density $\rho(\xi)\equiv\rho>0$, we have $\Ls^2_{\mathcal{H}}(0,\infty)=\Ls^2_{|\Delta|}(0,\infty)=\Ls^2(0,\infty)$ algebraically with equivalent norms. As $S$, $B$ and $C$ multiply with constant matrices the assumptions of Theorem \ref{PHS-THM} are satisfied and thus $A_{\mathcal{H}}\colon D(A_{\mathcal{H}})\rightarrow \Ls^2_{\mathcal{H}}(0,\infty)$ generates a $\Cnull$-semigroup if and only if \eqref{COND-1}, i.e., $w_1\gamma\not=w_2T$, holds.

\medskip

2.~Also under the standard assumption of, e.g., \cite{AJ2014, GZM2005, JMZ2015, JZ}, namely
\begin{equation}\label{STANDARD-ASS}
\exists\:m,\,M>0\;\forall\:\xi\in[0,\infty),\,\zeta\in\mathbb{C}^n \colon m|\zeta|^2\leqslant \zeta^{\star}\mathcal{H}(\xi)\zeta\leqslant M|\zeta|^2
\end{equation}
one can see that Theorem \ref{PHS-THM} is applicable, if we additionally assume that $\rho'$ and $T'$ are bounded. Notice that in \cite{AJ2014, GZM2005, JMZ2015, JZ} the estimate in \eqref{STANDARD-ASS} comes \textquotedblleft{}for free\textquotedblright{} if $\mathcal{H}$ is continuous, since in these papers $\xi\in[a,b]$ is considered. In our setting of a non-compact spacial domain this is not the case. Condition \eqref{STANDARD-ASS} implies that $\Delta$ satisfies, with different constants, the same estimate. From this it follows that $\Ls^2_{\mathcal{H}}(0,\infty)=\Ls^2_{|\Delta|}(0,\infty)=\Ls^2(0,\infty)$ are equal algebraically with equivalent norms. The boundedness of $\rho'$ and $T'$ guarantees that $B$, $C\colon \Ls^2(0,\infty)\rightarrow \Ls^2(0,\infty)$ are bounded. Therefore, the question of generation again reduces to \eqref{COND-1} and $A_{\mathcal{H}}\colon D(A_{\mathcal{H}})\rightarrow \Ls^2_{\mathcal{H}}(0,\infty)$ generates a $\Cnull$-semigroup if and only if $w_1\gamma(0)\not=w_2T(0)$ holds.

\medskip

3.~Finally, we want to give an explicit example where $\rho$ and $T$ are bounded, but \textit{not bounded away from zero}, and nevertheless the assumptions of Theorem \ref{PHS-THM} can be verified. Notice that below the Hamiltonian $\mathcal{H}$ does neither satisfy the lower nor the upper estimate in \eqref{STANDARD-ASS}. Let from now on $\rho(\xi)=1/\xi$ and $T(\xi)=1/\xi^3$ for $\xi\geqslant1$. For $\xi\in[0,1)$ we define $\rho(\xi)$ and $T(\xi)$ in a way that $\rho$, $T\colon[0,\infty)\rightarrow(0,\infty)$ are continuously differentiable. For the conditions to be checked, i.e., Theorem \ref{PHS-THM}(a)--(d), then the compact part $[0,1]$ of the domain can be neglected. For $\xi\geqslant1$ we have $\gamma(\xi)=1/\xi$, $\delta(\xi)=-2/\xi^5$ and $\sigma(\xi)=-1/\xi^2$. From this we see that
$$
\mathcal{H}=\begin{bmatrix}\xi&0\\0&1/\xi^3\end{bmatrix}
,\;\;|\Delta|=\begin{bmatrix}1/\xi&0\\0&1/\xi\end{bmatrix}
,\;\;B=\begin{bmatrix}0&-1/\xi^2\\1/\xi^2&0\end{bmatrix},
\text{ and }\; C=\frac{1}{2}\begin{bmatrix} -1/\xi  & 1/\xi^3\\ -1/\xi & -1/\xi^3 \end{bmatrix}
$$
holds for $\xi\geqslant1$ and we observe that Theorem \ref{PHS-THM}(a) holds. We select $K_1\geqslant1$ such that
$$
\|Bg\|_{\Ls^2_{|\Delta|}(0,\infty)}^2=\int_0^{\infty}g^{\star}B^{\star}|\Delta|Bg\dd\lambda\leqslant K_1\int_0^{\infty}g^{\star}|\Delta|g\dd\lambda=K_1\|g\|_{\Ls^2_{|\Delta|}(0,\infty)}^2
$$
holds for all $g\in \Ls^2_{|\Delta|}(0,\infty)$. Furthermore, we select $K_2\geqslant1$ such that
$$
\|Cx\|_{\Ls^2_{|\Delta|}(0,\infty)}^2=\int_0^{\infty}x^{\star}C^{\star}|\Delta|Cx\dd\lambda\leqslant K_2\int_0^{\infty}x^{\star}\mathcal{H}x\dd\lambda=K_2\|x\|_{\Ls^2_{\mathcal{H}}(0,\infty)}^2
$$
holds for all $x\in \Ls^2_{\mathcal{H}}(0,\infty)$. This is possible since
$$
B^{\star}|\Delta|B=\begin{bmatrix}1/\xi^5&0\\0&1/\xi^5\end{bmatrix}\;\text{ and }\; C^{\star}|\Delta|C=\frac{1}{2}\begin{bmatrix}1/\xi^2&0\\0&1/\xi^5\end{bmatrix}
$$
holds for $\xi\geqslant1$. This establishes that Theorem \ref{PHS-THM}(b)--(c) is satisfied. Finally we see that
$$
\|Sx\|_{\Ls^2_{|\Delta|}(0,\infty)}^2=\int_0^{\infty}x^{\star}S^{\star}|\Delta|Sx\dd\lambda=\frac{1}{2}\int_0^{\infty}x^{\star}\mathcal{H}x\dd\lambda=\frac{1}{2}\|x\|_{\Ls^2_{\mathcal{H}}(0,\infty)}^2
$$
holds since we have
$$
S^{\star}|\Delta|S=\begin{bmatrix}(2\gamma)^{-1}&0\\0&\rho^2\gamma/2\end{bmatrix}=
\frac{1}{2}\begin{bmatrix}\xi&0\\0&1/\xi^3\end{bmatrix}=
\frac{1}{2}\mathcal{H}
$$
for $\xi\geqslant1$. This implies that also Theorem \ref{PHS-THM}(b) is satisfied. It follows that the corresponding operator $A_{\mathcal{H}}\colon D(A_{\mathcal{H}})\rightarrow \Ls^2_{\mathcal{H}}(0,\infty)$ generates a $\Cnull$-semigroup if and only if $w_1\gamma(0)\not=w_2T(0)$ holds.

\medskip

There is a rich literature treating the wave equation in the case of variable coefficients, see, e.g., Todorova, Yordanov \cite{Todorova} and the references therein, deriving explicit estimates of the energy norm of solutions. Using our explicit formulas for the semigroups it is not hard to see that for constant $\rho$ and $T$ the semigroup corresponding to \eqref{VIB-1} is not strongly stable in the sense of \cite[Definition V.1.1]{EN}. Further stability results, treating the case of non-constant coefficients, will be contained in a forthcoming paper.
\end{ex}

\bigskip

\footnotesize

{\sc Acknowledgements.} The authors would like to thank the referee for her/his very careful review and the insightful feedback which helped to improve the article significantly.

\bigskip
\bigskip

\normalsize

\bibliographystyle{amsplain}

\providecommand{\bysame}{\leavevmode\hbox to3em{\hrulefill}\thinspace}
\providecommand{\MR}{\relax\ifhmode\unskip\space\fi MR }
\providecommand{\MRhref}[2]{%
  \href{http://www.ams.org/mathscinet-getitem?mr=#1}{#2}
}
\providecommand{\href}[2]{#2}

\end{document}